\documentclass[a4paper,reqno]{amsart}

\pdfoutput=1

\usepackage{nicefrac,amsmath,amssymb,amsthm,url,verbatim,enumitem,stmaryrd,mathtools,microtype,bbm,scalerel,thmtools}
\usepackage[utf8]{inputenc}
\usepackage[T1]{fontenc}
\usepackage{libertine}
\usepackage{dsfont}
\usepackage[libertine,cmintegrals,cmbraces]{newtxmath}
\usepackage[dvipsnames,svgnames,x11names]{xcolor}
\definecolor{darkred}{RGB}{139,0,0}
\definecolor{darkblue}{RGB}{0,0,139}
\definecolor{darkgreen}{RGB}{0,100,0}
\usepackage[mathscr]{euscript}
\usepackage[pagebackref]{hyperref}
\hypersetup{
	colorlinks   = true, 
	urlcolor     = darkred, 
	linkcolor    = darkred, 
	citecolor   = darkgreen}
\usepackage{tikz}
\usepackage{tikz-cd}
\usetikzlibrary{decorations.markings,patterns,arrows.meta}
\tikzset{>={Straight Barb[scale=0.8]}}
\tikzcdset{
	arrow style=tikz,
	diagrams={>={Straight Barb[scale=0.8]}}
}
\usepackage[capitalise,noabbrev]{cleveref}

\calclayout

\AtBeginDocument{%
   \def\MR#1{}
}

\setenumerate[1]{label=(\roman*),}
\setitemize[1]{label=\raisebox{0.25ex}{\tiny$\bullet$}}

\setlist[enumerate]{leftmargin=*}
\setlist[itemize]{leftmargin=*}

\newcommand\smallsquare{{\mathbin{\text{\raise0.17ex\hbox{\scalebox{.7}{$\blacksquare$}}}}}}

\newcommand{\circled}[1]{\raisebox{.5pt}{\textcircled{\raisebox{-.9pt} {#1}}}}

\newcommand{\plus}{\ensuremath{{\scaleobj{0.8}{+}}}}

\newcommand{\SO}{\ensuremath{\mathrm{SO}}}

\newcommand{\BO}{\ensuremath{\mathrm{BO}}}
\newcommand{\BSO}{\ensuremath{\mathrm{BSO}}}

\newcommand{\Fr}{\ensuremath{\mathrm{Fr}}}
\newcommand{\Diff}{{\ensuremath{\mathrm{Diff}}}}

\newcommand{\Map}{\ensuremath{\mathrm{Map}}}

\newcommand{\Emb}{\ensuremath{\mathrm{Emb}}}

\newcommand{\BTop}{\ensuremath{\mathrm{BTop}}}
\newcommand{\BSTop}{\ensuremath{\mathrm{BSTop}}}

\newcommand{\interior}{\ensuremath{\mathrm{int}}}

\newcommand{\inc}{\ensuremath{\mathrm{inc}}}

\newcommand{\Fun}{\ensuremath{\mathrm{Fun}}}

\newcommand{\PSh}{\mathrm{PSh}}

\newcommand{\cl}{\mathrm{cl}}

\DeclareMathAlphabet{\mathpzc}{OT1}{pzc}{m}{it}
\newcommand{\catsingle}[1]{\ensuremath{\mathscr{#1}}}
\newcommand{\cat}[1]{\ensuremath{\mathsf{#1}}}
\newcommand{\icat}[1]{\ensuremath{\mathscr{#1}}}
\newcommand{\half}{{\nicefrac{1}{2}}}

\newcommand{\midman}{{\shortparallel}}

\newcommand{\oH}{\ensuremath{\mathrm{H}}}

\newcommand{\oO}{\ensuremath{\mathrm{O}}}

\newcommand{\bfH}{\ensuremath{\mathbf{H}}}

\newcommand{\bfR}{\ensuremath{\mathbf{R}}}
\newcommand{\bfZ}{\ensuremath{\mathbf{Z}}}

\newcommand{\bfS}{\ensuremath{\mathbf{S}}}

\newcommand{\cD}{\ensuremath{\catsingle{D}}}

\newcommand{\cM}{\ensuremath{\catsingle{M}}}

\newcommand{\cO}{\ensuremath{\catsingle{O}}}
\newcommand{\cP}{\ensuremath{\catsingle{P}}}

\newcommand{\cS}{\ensuremath{\catsingle{S}}}

\newcommand{\colim}{\ensuremath{\mathrm{colim}}}

\newcommand{\ra}{\rightarrow}
\newcommand{\lra}{\longrightarrow}
\newcommand{\lla}{\longleftarrow}
\newcommand{\xra}[1]{\xrightarrow{#1}}
\newcommand{\xsra}[1]{\overset{#1}{\rightarrow}}
\newcommand{\xlra}[1]{\overset{#1}{\longrightarrow}}

\newcommand{\BSp}{\mathrm{BSp}}

\newcommand{\Sp}{\mathrm{Sp}}
\newcommand{\GL}{\mathrm{GL}}

\newcommand{\Aut}{\mathrm{Aut}}

\newcommand{\Hom}{\mathrm{Hom}}

\newcommand{\coker}{\mathrm{coker}}
\newcommand{\bP}{\mathrm{bP}}
\newcommand{\Ext}{\mathrm{Ext}}

\renewcommand{\Top}{\mathrm{Top}}
\newcommand{\ext}{\mathrm{ext}}
\newcommand{\sgn}{\sigma}
\newcommand{\im}{\mathrm{im}}

\newcommand{\id}{\mathrm{id}}

\newcommand{\pr}{\mathrm{pr}}

\newcommand{\fib}{\mathrm{fib}}

\newcommand{\op}{\mathrm{op}}

\newcommand{\forget}{\mathrm{forget}}
\newcommand{\Env}{\ensuremath{\mathrm{Env}}}

\newcommand{\IntMap}{{\mathrm{map}}}
\newcommand{\Th}{{\mathrm{Th}}}

\newcommand{\rev}{{\mathrm{rev}}}

\newcommand{\ob}{{\mathrm{ob}}}

\newcommand{\Disc}{\ensuremath{\cat{Disc}}}
\newcommand{\DiscInf}{\ensuremath{\icat{D}\mathrm{isc}}}

\newcommand{\ManInf}{\ensuremath{\icat{M}\mathrm{fd}}}

\newcommand{\ncBordInf}{\ensuremath{\mathrm{nc}\icat{B}\mathrm{ord}}}

\newcommand{\Cat}{\ensuremath{\icat{C}\mathrm{at}}}

\newcommand{\AAlg}{\ensuremath{\mathrm{AAlg}}}
\newcommand{\CatInf}{\ensuremath{\icat{C}\mathrm{at}_\infty}}

\newcommand{\glue}{\ensuremath{\mathrm{glue}}}
\newcommand{\tw}{\ensuremath{\mathrm{tw}}}
\newcommand{\fr}{\ensuremath{\mathrm{fr}}}
\newcommand{\Poinc}{\ensuremath{\icat{P}\mathrm{oinc}}}

\newcommand{\Wh}{\ensuremath{\mathrm{Wh}}}
\newcommand{\LMod}{\ensuremath{\mathrm{LMod}}}
\newcommand{\RMod}{\ensuremath{\mathrm{RMod}}}

\newcommand{\Mod}[1]{\ (\mathrm{mod}\ #1)}

\newtheorem{bigthm}{Theorem}

\newtheorem{thm}{Theorem}[section]

\newtheorem{lem}[thm]{Lemma}

\newtheorem{question}[thm]{Question}

\theoremstyle{definition}

\newtheorem*{nconvention}{Convention}
\theoremstyle{remark}
\newtheorem{ex}[thm]{Example}

\newtheorem{rem}[thm]{Remark}
\newtheorem*{nrem}{Remark}

\newtheorem{construction}[thm]{Construction}

\begin{document}

\title{Framed configuration spaces and exotic spheres}

\author{Manuel Krannich}
\address{Department of Mathematics, Karlsruhe Institute of Technology, 76131 Karlsruhe, Germany}
\email{krannich@kit.edu}

\author{Alexander Kupers}
\address{Department of Computer and Mathematical Sciences, University of Toronto Scarborough, 1265 Military Trail, Toronto, ON M1C 1A4, Canada}
\email{a.kupers@utoronto.ca}

\author{Fadi Mezher}
\address{Department of Mathematics, Karlsruhe Institute of Technology, 76131 Karlsruhe, Germany}
\email{fadi.mezher@kit.edu}


\begin{abstract}
We determine when an exotic sphere $\Sigma$ of dimension $d\not{\equiv }1\Mod{4}$ can be detected through the homotopy type of its truncated $\DiscInf$-presheaf. The latter records the diagram of framed configuration spaces of bounded cardinality in $\Sigma$ with natural point-forgetting and -splitting maps between them, and it gives rise to the finite stages in Goodwillie--Weiss' embedding calculus tower. Our proof involves three ingredients that could be of independent interest: a gluing result for $\DiscInf$-presheaves of manifolds divided into two codimension zero submanifolds, a version of Atiyah duality in the context of $\DiscInf$-presheaves, and a computation of the finite residual of the mapping class group of the connected sums $\sharp^g(S^{2k+1}\times S^{2k+1})$.
\end{abstract}

\maketitle

\tableofcontents

\section{Introduction} How ``much'' of a closed $d$-dimensional manifold is seen by the homotopy types of its configuration spaces? There are several ways to make this question precise, and to some of them partial answers have been given \cite{Levitt,AouinaKlein,LongoniSalvatore,AroneSzymik,KnudsenKupers,KKSDisc}. In this work we answer an instance of this question in the case of exotic spheres. 

To explain the result, we write $\ManInf_d$ for the $\infty$-category with smooth $d$-dimensional manifolds $M$ as objects and spaces of smooth embeddings $\Emb(M,N)$ as morphisms. This has a full subcategory $\DiscInf_d\subseteq \ManInf_d$ spanned by the manifolds $S \times \bfR^d$ for finite sets $S$. To a smooth $d$-manifold $M$, we can associate a presheaf $E_M$ on $\DiscInf_d$ with values in the $\infty$-category of spaces $\cS$, given by 
\begin{equation}\label{equ:em-def}\smash{\DiscInf_d^\op\ni S \times \bfR^d\xmapsto{E_M}\Emb(S \times \bfR^d,M)\in \cS,}\end{equation}
and this construction extends---via postcomposition---to a functor out of $\ManInf_d$
\[\smash{\ManInf_d\ni M\xmapsto{\ E\ }E_M\in \PSh(\DiscInf_d)}\]
with values in the $\infty$-category $\PSh(\DiscInf_d)\coloneq \Fun(\DiscInf_d^\op,\cS)$ of presheaves. The individual values of  the presheaf $E_M$ at objects $\smash{S \times \bfR^d}$ are equivalent to the \emph{framed configuration spaces},
\begin{equation}\label{equ:framed-configuration}\smash{\Emb(S \times \bfR^d,M)\simeq F^{\fr}_S(M)\coloneqq\big\{(m_s,T_{m_s}M\overset{\phi_s}\cong \bfR^d)_{s\in S}\in \Fr(TM)^S\mid m_i\neq m_j\text{ for }i\neq j\big\},}\end{equation}
so $E_M\in  \PSh(\DiscInf_d)$ encodes, roughly speaking, the homotopy types of all framed configuration spaces in $M$ together with natural maps between them (e.g.\,the value of $E_M$ at  $(\iota\times \id_{\bfR^d})\colon S\times\bfR^d\hookrightarrow S'\times \bfR^d$ for an injection $\iota\colon S\hookrightarrow S'$ corresponds under \eqref{equ:framed-configuration} to forgetting some of the points). The equivalence class of the presheaf $E_M\in \PSh(\DiscInf_d)$ is a diffeomorphism-invariant of $M$ from which many other invariants can be extracted, including the \emph{factorisation homology} of $M$ with coefficients in any $\BO(d)$-framed $E_d$-algebra \cite{Salvatore, AyalaFrancisTop, LurieHA} or the \emph{embedding calculus tower} for spaces of embeddings into or out of $M$ \cite{WeissImmersion,BdBWSheaf}. In fact, for $d\ge5$, there is still no known example of two closed non-diffeomorphic $d$-manifolds $M$ and $N$ such that $E_M$ and $E_N$ are equivalent in $\PSh(\DiscInf_d)$\footnote{For $d=4$ such examples are known by \cite[Theorem B]{KnudsenKupers}.}.

There are several variants of the presheaf $E_M$ one can consider, e.g.~one for topological manifolds $M$ based on a variant of the $\infty$-category $\smash{\DiscInf_{d}}$ involving topological embeddings, or one for manifolds $M$ equipped with a tangential structure: for instance, if $M$ comes equipped with an orientation one can consider the analogue  $\smash{\DiscInf^\plus_{d}}$ of $\smash{\DiscInf_{d}}$ where one restricts to orientation-preserving embeddings, and the presheaf $\smash{E_M\in \PSh(\DiscInf^\plus_{d})}$ defined as in \eqref{equ:em-def} but using spaces of orientation-preserving embeddings. As another variant, instead of considering presheaves on the full $\infty$-category $\DiscInf_d$, one may for fixed $k\ge 1$ restrict to the full subcategory $\DiscInf_{d,\le k}\subseteq \DiscInf_d$ spanned by $S\times \bfR^d$ for $|S| \leq k$ and consider the presheaf $E_M\in \PSh(\DiscInf_{d,\le k})$ obtained from $E_M$ by restriction; this is called the \emph{$k$-truncated $\DiscInf$-presheaf} of $M$. In terms of diagrams of framed configuration spaces, passing to the truncated setting amounts to imposing an upper bound on the number of points.

\subsection*{The main result}In this work, we study the question of which homotopy $d$-spheres can be distinguished by their truncated $\DiscInf^{\plus}$-presheaves. Recall that a \emph{homotopy $d$-sphere} is a closed smooth $d$-manifold $\Sigma$ that is homotopy equivalent to a standard $d$-sphere, and thus---by the solution to the topological Poincaré conjecture---also homeomorphic to it. To state our main result, we write  $M\sharp N$ for the connected sum of two oriented $d$-manifolds $M$ and $N$, and $\overline{M}$ for $M$ equipped with the opposite orientation.

\begin{bigthm}\label{bigthm:exotic-spheres}For oriented homotopy $d$-spheres $\Sigma_0$ and $\Sigma_1$ with $d\not\equiv1\Mod{4}$ and $1< k < \infty$, we have 
	$E_{\Sigma_0} \simeq E_{\Sigma_1}$ in $\PSh(\DiscInf^{\plus}_{d,\leq k})$ if and only if $\Sigma_0 \sharp \overline{\Sigma}_1$ bounds a compact parallelisable manifold.
\end{bigthm}

\begin{nrem}\ 
\begin{enumerate}
\item \cref{bigthm:exotic-spheres} answers \cite[Question 5.5]{KnudsenKupers} in many cases and its proof also yields a partial answer to Question 5.6 loc.cit..
\item The question of which homotopy $d$-spheres bound compact parallelisable manifold has been studied extensively in the past and features in Kervaire--Milnor's classification of homotopy spheres \cite{KervaireMilnor}. For example, from the latter one can deduce that there are $16256$ oriented homotopy $15$-spheres up to orientation-preserving diffeomorphism, of which precisely half bound compact parallelisable manifolds. By \cref{bigthm:exotic-spheres}, this implies that also precisely half of them have the property that they can be distinguished from the standard $15$-sphere by their $k$-truncated $\DiscInf^{\plus}$-presheaf for any $1<k<\infty$.
\item For $d$-manifolds $M$ and $N$, the mapping spaces $\Map_{\PSh(\Disc_d,\le k)}(E_M,E_N)$ are equivalent to the $k$th stage $T_k\Emb(M,N)$ in the embedding calculus tower \cite{BritoWeissSheaves}. From this point of view, \cref{bigthm:exotic-spheres} answers the question which homotopy $d$-spheres can be detected through the finite stages of the embedding calculus tower if $d\not\equiv1\Mod{4}$.
\end{enumerate}
\end{nrem} 

\subsection*{On the assumptions}\label{sec:open}We comment on the assumptions on $d$ and $k$ in \cref{bigthm:exotic-spheres}:

\subsubsection*{The case $k=1$}Two oriented $d$-manifolds $M$ and $N$ have equivalent $k$-truncated $\DiscInf^{\plus}$-presheaves for $k=1$ if and only if there is a \emph{tangential homotopy equivalence} between them, i.e.\,a homotopy equivalence $\varphi\colon M\ra N$ covered by an orientation-preserving fibrewise isomorphism of the tangent bundles \cite[Proposition 5.10]{KKoperadic}. Since it is known that any two oriented homotopy $d$-spheres $\Sigma_0$ and $\Sigma_1$ are tangentially homotopy equivalent (see e.g.\ \cite[Lemma 1.1]{RayPedersen}), one gets that $E_{\Sigma_0}\simeq E_{\Sigma_1}$ in $\PSh(\DiscInf^{\plus}_{d,\leq k})$ for $k=1$ and \emph{all} $\Sigma_0$ and $\Sigma_1$. 

\begin{nrem}As any homotopy equivalence between homotopy spheres is homotopic to a homeomorphism, one can even choose a tangential homotopy equivalence $\varphi\colon \Sigma_0\ra \Sigma_1$ that is a homeomorphism. The latter induces homotopy equivalences $F_S^\fr(\Sigma_0)\simeq F_S^\fr(\Sigma_1)$ between all framed configuration spaces, so in view of \eqref{equ:framed-configuration}, this  shows that the truncated $\DiscInf$-presheaves $E_{\Sigma_0}$ and $E_{\Sigma_1}$ can never be distinguished by considering any of their individual values.\end{nrem}
		 
\subsubsection*{The case $d\equiv 1\Mod{4}$} Our proof of the ``only if'' direction in \cref{bigthm:exotic-spheres} also goes through in the excluded dimensions $d \equiv 1 \Mod 4$ (see \cref{sec:coker-j-spheres}). The ``if'' direction however (i.e.\,whether $\smash{E_{\Sigma_0}\simeq E_{\Sigma_1}}$ if $\smash{\Sigma_0\sharp\overline{\Sigma_1}}$ bounds a compact parallelisable manifold) does not. By the solution of the Kervaire invariant one problem, it is known that for $d=2^\ell-3$ with $\ell\le 7$, only the standard sphere bounds a compact parallelisable manifold, so there is nothing to show, but in all other dimensions $d\equiv1\Mod{4}$, there is a single nontrivial oriented homotopy sphere that bounds a compact parallelisable manifold, called the \emph{Kervaire sphere}. This leads us to ask:

\begin{question}\label{qu:kervaire}For which $m\ge2$ and $2\le k< \infty$ is there an equivalence $E_{\Sigma_{K}}\simeq E_{S^{4m+1}}$ in $\PSh(\DiscInf^{\plus}_{d,\le k})$ where  $\Sigma_{K}$ is the $(4m+1)$-dimensional Kervaire sphere?
\end{question}

\subsubsection*{The case $k=\infty$}In the excluded case $k=\infty$, the ``only if'' direction of \cref{bigthm:exotic-spheres} follows from the case $1<k<\infty$, since an equivalence of nontruncated presheaves induces by restriction one of all their truncations. However, the ``if''-direction does not follow: similarly to how there are non-equivalent spaces whose finite Postnikov truncations are all equivalent \cite{AdamsExample}, there may be non-equivalent presheaves $X$ and $Y$ in $\PSh(\DiscInf^{\plus}_d)$ that are equivalent in $\smash{\PSh(\DiscInf_{d,\le k}^\plus)}$ for all $k<\infty$. \cref{bigthm:exotic-spheres} thus leaves open the question for which homotopy $d$-spheres $\Sigma_0$ and $\Sigma_1$ such that the connected sum $\Sigma_0\sharp\overline{\Sigma}_1$ bounds a compact parallelisable manifold there exists an equivalence $E_{\Sigma_0}\simeq E_{\Sigma_1}\in\PSh(\DiscInf^{\plus}_{d})$, \emph{without truncation}. In particular we ask:

\begin{question}\label{qu:nontrun}For which oriented homotopy spheres $\Sigma$ that bound a compact parallelisable manifold does there exist an equivalence $E_{\Sigma}\simeq E_{S^{d}}$ in $\PSh(\DiscInf^{\plus}_{d})$? 
\end{question}

\subsection*{Ingredients in the proof}The proof of \cref{bigthm:exotic-spheres} involves three ingredients that could be of independent interest:
\begin{enumerate}
\item A description of $\DiscInf$-presheaves of manifolds that are decomposed into two codimension $0$ submanifolds that intersect in their common boundary (see \cref{thm:gluing}).
\item A proof that the Atiyah duality equivalence $D(\Sigma_+^\infty M)\simeq \Th(-TM)$ for closed $d$-manifolds is natural in equivalences of $\DiscInf$-presheaves (see \cref{thm:collapse-zigzag} and \cref{sec:normal-invariants}), based on a variant of a construction due to Naef--Safranov \cite[Section 4.2]{NaefSafronov} (see \cref{rem:naefsafronov}). This also suggests a relation between algebraic $L$-theory and the $\DiscInf$-structure spaces as introduced in \cite{KKSDisc} (see \cref{sec:sdisc-ltheory}).
\item A computation of the finite residual (the intersection of all finite-index normal subgroups) of the group of isotopy classes of orientation-preserving diffeomorphisms of the iterated connected sums $\sharp^g(S^{2m+1}\times S^{2m+1})$ (see \cref{thm:fr-prelim} and \cref{sec:fr}).
\end{enumerate}

\subsection*{Acknowledgments}FM would like to thank Søren Galatius for the many helpful discussions regarding this work. AK and MK would like to thank Oscar Randal-Williams for useful conversations. MK and FM acknowledge funding from the European Union through an ERC grant (MaFC, 101221003). AK acknowledges the support of the Natural Sciences and Engineering Research Council of Canada (NSERC) [funding reference number 512156 and 512250]. FM is supported by the Danish National Research Foundation through the Copenhagen Center for Geometry and Topology (DNRF151).

\section{$\DiscInf$-presheaves of exotic spheres}\label{sec:main-ideas}
In this section, we prove \cref{bigthm:exotic-spheres} (or rather a strengthening of it) assuming three ingredients (Theorems\ \ref{thm:gluing-isotopy}, \ref{thm:fr-prelim}, and \ref{thm:collapse-zigzag}) which are established in the subsequent three sections.

\begin{nconvention}We work in the setting of $\infty$-categories throughout, so a ``category'' is always an $\infty$-category. Unless mentioned otherwise, manifolds and embeddings are always assumed to be smooth. Given a homotopy $d$-sphere $\Sigma$, we write $\Sigma\in \bP_{d+1}$ if $\Sigma$ bounds a compact parallelisable $(d+1)$-manifold. To treat the categories $\DiscInf_{d,\le k}$ and $\DiscInf_d$ from the introduction on the same footing, we write $\DiscInf_{d,\le \infty}\coloneqq \DiscInf_d$ (similarly for the oriented variant).
\end{nconvention}

\subsection{$\DiscInf$-presheaves of $\bP$-spheres}\label{sec:e-of-bp}The ``if'' direction of \cref{bigthm:exotic-spheres} follows from the following more general result by specialising to $M=S^d$:

\begin{thm}\label{thm:bpspheres} For a connected oriented $d$-manifold $M$ with  $d\not\equiv 1\Mod{4}$ and two oriented homotopy $d$-spheres $\Sigma_0,\Sigma_1$ with $\Sigma_0\sharp\overline{\Sigma_1}\in\bP_{d+1}$, there is for all $k<\infty$ an equivalence
\[E_{M\sharp \Sigma_0}\simeq E_{M\sharp \Sigma_1}\quad\text{in}\ \ \PSh(\DiscInf^\plus_{d,\le k}).\]
In particular, for all $\Sigma\in \bP_{d+1}$, we have $E_{M\sharp\Sigma}\simeq E_M$ in $\PSh(\DiscInf^\plus_{d,\le k})$ whenever $1\le k<\infty$. \end{thm}

The proof of \cref{thm:bpspheres} has two main ingredients---one homotopy-theoretic and one group-theoretic. The homotopy-theoretic one---which we prove in \cref{sec:gluing}---is the following criterion to show that two manifolds obtained by gluing together the same pair of manifolds along different diffeomorphisms of their boundaries have equivalent $\DiscInf$-presheaves:

\begin{thm}\label{thm:gluing-isotopy} Fix $d$-manifolds $M_0$ and $M_1$ and two diffeomorphisms $\varphi_0,\varphi_1\colon \partial M_0\ra \partial M_1$ between their boundaries. If for some fixed $1\le k\le \infty$ we have
\[\smash{[E_{\varphi_0}]=[E_{\varphi_1}]\in\pi_0\,\Map_{\PSh(\DiscInf_{d-1,\le k})}(E_{\partial M_0},E_{\partial M_1}),}\] then there is an equivalence
 \[\smash{E_{M_0\cup_{\varphi_0 }M_1}\simeq E_{M_0\cup_{\varphi_1 }M_1}\quad \text{in}\quad\PSh(\DiscInf_{d,\le k}).}\]
Moreover, the analogous statement holds when $M_0$ and $M_1$ are oriented, the $\varphi_i$ are orientation-reversing, and $\DiscInf_{d-1,\le k}$ and $\DiscInf_{d,\le k}$ are replaced by $\DiscInf^\plus_{d-1,\le k}$ and $\DiscInf^\plus_{d,\le k}$.
\end{thm}

The second ingredient in the proof of \cref{thm:bpspheres} is related to a certain group-theoretic difference between mapping class groups of manifolds and the analogue in the context of $\DiscInf$-presheaves, which was discovered in \cite{Mezher}. To explain this, recall that a discrete group $G$ is \emph{residually finite} if its \emph{finite residual}\[\textstyle{\fr(G) \coloneqq \big(\bigcap_{G' \trianglelefteq G \text{ with } [G':G] < \infty} G'\big) \trianglelefteq G}\] vanishes. It was observed in \cite{KRW-arithmetic} that there are 2-connected closed high-dimensional manifolds $M$ for which the group $\pi_0\,\Diff^{\plus}(M)$ of isotopy classes of orientation preserving diffeomorphisms is \emph{not} residually finite (more specifically this was shown for the iterated connected sums $W_g\coloneq \sharp^g (S^n\times S^n)$ for certain values of $g$ and $n$). On the contrary, the analogous group $\pi_0 \,\Aut_{\PSh(\DiscInf_{d,\le k})}(E_M)$ for the truncated $\DiscInf$-presheaf of $M$ \emph{is} residually finite for any $k<\infty$, as shown in \cite{Mezher}. The proof of \cref{thm:gluing-isotopy} will exploit this difference, based on an extension of the result from \cite{KRW-arithmetic} which we explain now. Recall that by Cerf's ``pseudoisotopy-implies-isotopy'' theorem and Smale's h-cobordism theorem, extending diffeomorphisms of $D^{d-1}$ along the inclusion of a hemisphere $D^{d-1}\subset S^{d-1}$ via the identity, and gluing two copies of $D^d$ together along a diffeomorphism of $S^{d-1}$, yields for $d\ge6$ isomorphisms of abelian groups
\begin{equation}\label{equ:tw-sphere}\smash{\pi_0\,\Diff_\partial(D^{d-1})\overset{\ext}{\cong}\pi_0\,\Diff^\plus(S^{d-1})\overset{\glue}{\cong}}\Theta_d.\end{equation} 
where $\pi_0\,\Diff_\partial(D^{d-1})$ is the group of isotopy classes of diffeomorphisms of a closed $(d-1)$-disc that pointwise fix the boundary and $\Theta_d$ is Kervaire--Milnor's group of oriented homotopy $d$-spheres \cite{KervaireMilnor}. For an oriented $(d-1)$-manifold $M$ with an embedded codimension $0$ disc $D^{d-1}\subset M$ compatible with the orientation, we write 
\[ \smash{\pi_0\,\Diff_\partial(D^{d-1})\xrightarrow{\ext_M} \pi_0\,\Diff^\plus_\partial(M)}\] for the morphism given by extending diffeomorphisms along $D^{d-1}\subset M$ by the identity. In these terms, the extension of the result of \cite{KRW-arithmetic} which we prove in \cref{sec:fr} is:

\begin{thm}\label{thm:fr-prelim}For $n\ge 3$ odd and $g\ge0$, setting $W_g\coloneq\sharp^g (S^n\times S^n)$, we have 
\[\fr\big(\pi_0\,\Diff^{\plus}(W_g)\big)=\begin{cases}\im\Big(\bP_{2n+2}\le\Theta_{2n+1}\overset{\eqref{equ:tw-sphere}}{\cong} \pi_0\,\Diff_\partial(D^{2n})\xra{\ext_{W_g}}\pi_0\,\Diff^{\plus}(W_g)\Big)&\text{if }g\ge2,\\
0&\text{if }g< 2.
\end{cases}\]
\end{thm}

\begin{rem}\,
\begin{enumerate}
\item The morphism $\ext_{W_g}$ turns out to be injective (see \cref{sec:fr-computation} \ref{enum:gamma-ext}), so in the first case of \cref{thm:fr-prelim} we have $\fr(\pi_0\,\Diff^{\plus}(W_g))\cong\bP_{2n+2}$.
\item It was shown in \cite[Theorem C]{Mezher} that there is an inclusion
\[\smash{\fr(\pi_0\,\Diff^+(M))\subset \im\big(\Theta_{d}\overset{\eqref{equ:tw-sphere}}{\cong} \pi_0\,\Diff_\partial(D^{d-1})\xra{\ext_{M}}\pi_0\,\Diff^+(M)\big)}\] for \emph{any} closed smooth $2$-connected manifold $M$ of dimension $d-1\ge6$.
\item Under the assumptions of \cref{thm:fr-prelim}, the mapping class group $\pi_0\,\Diff^{\plus}(W_g)$ has a complete algebraic description in terms of the integral symplectic group $\Sp_{2g}(\bfZ)$, by the main results of \cite{Krannich-mcg} which build on work of Kreck \cite{Kreck}. The proof of \cref{thm:fr-prelim} uses this algebraic description to explicitly compute the finite residual. In the excluded case $n$ even, a similarly complete description of the mapping class group $\pi_0\,\Diff^{\plus}(W_g)$ has not yet been established.
\end{enumerate}
\end{rem}

Assuming the two ingredients Theorems \ref{thm:gluing-isotopy} and \ref{thm:fr-prelim}, we finish the proof of \cref{thm:bpspheres}:

\begin{proof}[Proof of \cref{thm:bpspheres}]Throughout, we implicitly use the topological Poincaré conjecture to identify the topological manifold underlying a homotopy sphere with the standard sphere.  

We first assume $d\le 6$. For $d\le 3$ or $d=5,6$, there is nothing to show since there are no nontrivial homotopy spheres. For $d=4$ there may be, but we show that $E_{M\sharp \Sigma}\simeq E_M$ in $\PSh(\DiscInf^{\plus}_d)$ for \emph{all} oriented homotopy $d$-spheres $\Sigma$ (and hence also in $\PSh(\DiscInf^{\plus}_{d,\leq k})$ for any $k$, by restriction). This is closely related to \cite[Theorem B]{KnudsenKupers}. By smoothing theory for embedding calculus  \cite[Theorem 5.21 (ii)]{KKoperadic}, it suffices to show that the two lifts of the topological tangent bundle $M\ra \BSTop(4)$ along $\BSO(4)\ra \BSTop(4)$ induced by the smooth structure of $M$ and of $M\sharp \Sigma$ are homotopic as lifts. For that it is enough to show that the two lifts of $D^4\ra \BSTop(4)$ induced by the smooth structure of $D^4$ and of $D^4\sharp\Sigma$ are homotopic as lifts, relative to the given lift on $\partial D^4$ induced by the standard smooth structure. But the obstruction for the existence of such a homotopy of lifts lies in $\pi_4(\Top(4)/\oO(4))$ and this group vanishes, since $\pi_4(\Top(4)/\oO(4))\cong\pi_4(\Top/\oO)$ by \cite[Theorem 8.7A]{FreedmanQuinn} and $\pi_4(\Top/\oO)$ vanishes by \cite[V.5.0 (5)]{KirbySiebenmann}.

\begin{figure}
	\begin{tikzpicture}
		\begin{scope}[xshift=-4cm]
		\draw (0,0) -- (2,0) to[in=0,out=0,looseness=2] (2,2) -- (0,2) to[in=180,out=180,looseness=2] cycle;
		\draw [thick,darkgreen] plot [smooth] coordinates {(1,0) (.5,.5) (1.5,1.5) (1,2)};
		\node at (-.25,1) {$M_0$};
		\node at (2.25,1) {$M_1$};
		\node at (.65,1) [darkgreen] {$P$};
		\end{scope}
		\node at (0,1) {$=$};
		\begin{scope}[xshift=2cm]
		\draw (2.5,0) -- (3.5,0) to[in=0,out=0,looseness=2] (3.5,2) -- (2.5,2);
		\draw (1,2) -- (0,2) to[in=180,out=180,looseness=2] (0,0) -- (1,0);
		\draw [thick,darkgreen] plot [smooth] coordinates {(1,0) (.5,.5) (1.5,1.5) (1,2)};
		\draw [thick,darkgreen,xshift=1.5cm] plot [smooth] coordinates {(1,0) (.5,.5) (1.5,1.5) (1,2)};
		\node at (-.25,1) {$M_0$};
		\node at (3.75,1) {$M_1$};
		\node at (1.75,1) [darkgreen] {$\underset{\cong}{\overset{\id_P}{\lra}}$};
		\end{scope}
		
		\node at (3.75,-.75) {\rotatebox{270}{$\leadsto$}};
		
		\begin{scope}[xshift=2cm,yshift=-3.5cm]
			\draw (2.5,0) -- (3.5,0) to[in=0,out=0,looseness=2] (3.5,2) -- (2.5,2);
			\draw (1,2) -- (0,2) to[in=180,out=180,looseness=2] (0,0) -- (1,0);
			\draw [thick,darkgreen] plot [smooth] coordinates {(1,0) (.5,.5) (1.5,1.5) (1,2)};
			\draw [thick,darkgreen,xshift=1.5cm] plot [smooth] coordinates {(1,0) (.5,.5) (1.5,1.5) (1,2)};
			\node at (-.25,1) {$M_0$};
			\node at (3.75,1) {$M_1$};
			\fill [white,opacity=.8] (.75,.5) rectangle (2.75,1.5);
			\node at (1.75,1) [darkgreen] {$\underset{\cong}{\overset{\ext_P(\tw(\Sigma)))}{\lra}}$};
		\end{scope}
		
		\node at (-3,-.75) {\rotatebox{270}{$\leadsto$}};
		
		\begin{scope}[xshift=-4cm,yshift=-3.5cm]
			\draw (0,0) -- (2,0) to[in=0,out=0,looseness=2] (2,2) -- (0,2) to[in=180,out=180,looseness=2] cycle;
			\draw [thick,darkgreen] plot [smooth] coordinates {(1,0) (.5,.5) (1.5,1.5) (1,2)};
			\node at (-.25,1) {$M_0$};
			\node at (2.25,1) {$M_1$};
			\fill [white,opacity=.75] (1,1) circle (.45cm);
			\draw [dashed] (1,1) circle (.45cm);
			\node [fill=white] at (1,1) {$\Sigma$};
		\end{scope}
		\node at (0,-2.5) {$=$};
	\end{tikzpicture}
	\caption{Constructing $M \# \Sigma$ by modifying the gluing diffeomorphism in a splitting $M=M_0\cup_PM_1$ along a codimension zero submanifold $P$.}
	\label{fig:splitting}
\end{figure}

Turning to the case $d\ge7$, we begin with a general observation: Given a decomposition $M=M_0\cup M_1$ of an oriented $d$-manifold $M$ into two codimension $0$ submanifolds $M_0,M_1\subset N$ that intersect in their common boundary $P\coloneq\partial M_0=\partial M_1$, then for any embedded disc $D^{d-1}\subset P$ compatible with the orientation, there is an orientation-preserving diffeomorphism $M\sharp \Sigma \cong M_0\cup_{\ext_P(\tw(\Sigma))}M_1$ where $\tw \colon \Theta_d \to \pi_0\,\Diff_\partial(D^{d-1})$ is the inverse of the isomorphism from \eqref{equ:tw-sphere} (see \cref{fig:splitting}). Combining this with \cref{thm:gluing-isotopy}, when given oriented homotopy $d$-spheres $\Sigma_0$ and $\Sigma_1\in\Theta_d$, to show $E_{M\sharp \Sigma_0}\simeq E_{M\sharp \Sigma_1}$ in $\PSh(\DiscInf_{d,\le k}^\plus)$ for some fixed $1\le k\le \infty$ it suffices to show that $\smash{\tw(\Sigma_0\sharp\overline{\Sigma_1})=\tw(\Sigma_0)\circ \tw(\Sigma_1)^{-1}}$ lies in the kernel of the composition of $\ext_P$ with the morphism $\smash{E\colon \pi_0\,\Diff^\plus(P)\ra \pi_0\,\Aut_{\PSh(\DiscInf^\plus_{d-1,\le k})}(E_{P})}$. Since $\bP_{d+1}=0$ for $d$ even by \cite[Theorem 5.1, Lemma 2.3]{KervaireMilnor} and we assumed $d\not\equiv 1\Mod{4}$, this shows that in order to prove \cref{thm:bpspheres}, it suffices to find for $d=2n+1$ with $n\ge3$ odd a decomposition $M=M_0\cup M_1$ as above, such that the following composition is trivial for $1\le k<\infty$:
\begin{equation}\label{equ:composition-gluing}\smash{\bP_{2n+2}\le \smash{\Theta_{2n+1}\overset{\eqref{equ:tw-sphere}}{\cong}\pi_0\,\Diff_\partial(D^{2n})\xrightarrow{\ext_P} \pi_0\,\Diff^\plus(P)\xlra{E} \pi_0\,\Aut_{\PSh(\DiscInf^\plus_{d,\le k})}(E_{P})}.}
\end{equation}
 The decomposition we use is the following:  for any fixed $g\ge0$, choose an embedding of the iterated boundary connected sum $V_g\coloneq \natural^g D^{n+1}\times S^n$ into $M$ (such an embedding exists since $V_g$ embeds into a closed disc $D^{2n+1}$ which in turn embeds into any nonempty manifold), set $M_0$ to be the image of the embedding and $M_1$ to be the closure of its complement. The submanifolds $M_0$ and $M_1$ intersect in their common boundary which is diffeomorphic to $W_{g}= \partial(V_g)=\sharp^g S^n \times S^n$. It thus suffices to show that \eqref{equ:composition-gluing} is trivial for $P=W_g$ for \emph{some} value of $g\ge0$. This is true for any $g\ge2$, since the composition $\bP_{2n+2}\ra \pi_0\,\Diff^{\plus}(W_g)$ lands in the finite residual of $\pi_0\,\Diff^{\plus}(W_g)$ by \cref{thm:fr-prelim} (it is in fact equal to it, but we will not need this), so as finite residuals are preserved by group homomorphisms, the image of the composition \eqref{equ:composition-gluing} is for $g\ge2$ contained in the finite residual of $\smash{\pi_0\,\Aut_{\PSh(\DiscInf^\plus_{d,\le k})}(E_{W_g})}$. But the latter finite residual is trivial by \cite[Theorem 3.16]{Mezher}, so the claim follows.
\end{proof}

\subsection{$\DiscInf$-presheaves of $\coker(J)$-spheres}\label{sec:coker-j-spheres}We now turn on the ``only if'' direction of \cref{bigthm:exotic-spheres}. In dimensions $d\le 6$ any homotopy $d$-sphere bounds a compact parallelisable manifold (combine \cite[p.\,504]{KervaireMilnor} with \cite[Theorem 3]{KervaireHomology}), so there is nothing to show. In dimensions $d\ge7$, we use that by Kervaire--Milnor's exact sequence \cite{KervaireMilnor}, the condition $\Sigma_0\sharp\overline{\Sigma}_1\in\bP_{d+1}$ in \cref{bigthm:exotic-spheres} is equivalent to $\Sigma_0$ and $\Sigma_1$ having the same image under the morphism $\Theta_d\ra\coker(J)_d$ from loc.cit., so the ``only if'' direction of \cref{bigthm:exotic-spheres} follows from:

\begin{thm}\label{thm:2-truncated-cokerJ} For oriented homotopy $d$-spheres $\Sigma_0$ and $\Sigma_1$ whose associated oriented $2$-truncated  $\DiscInf$-presheaves $E_{\Sigma_0}, E_{\Sigma_1}\in \PSh(\DiscInf_{d,\le 2}^{\plus})$ are equivalent, we have \[[\Sigma_0]=[\Sigma_1]\in\coker(J)_d.\] In particular, if $[\Sigma]\neq0\in\coker(J)_d$, then $E_{\Sigma}\not\simeq E_{S^d}$ in $\PSh(\DiscInf_{d,\le 2}^{\plus})$.
\end{thm}

The main ingredient for the proof of \cref{thm:2-truncated-cokerJ} says, informally speaking, that the Atiyah duality equivalence for manifolds can be made natural in equivalences of $2$-truncated $\DiscInf$-presheaves. To explain the precise statement, recall that a $k$-truncated $\DiscInf$-presheaf $X\in\PSh(\DiscInf_{d,\le k})$ has for $k\ge1$ an underlying space, given by the orbits \[\lvert X\rvert\coloneq X(\bfR^d)/\Diff(\bfR^d)\in \cS\] of the action of $\smash{\Diff(\bfR^d)\simeq\Aut_{\PSh(\DiscInf_{d,\le k})}(\bfR^d)\simeq\oO(d)}$ on $X(\bfR^d)$ by functoriality, and this space comes with a $d$-dimensional vector bundle $\xi_X$ over it, classified by the map
\begin{equation}\label{equ:vb-over-x}\smash{\lvert X\vert ={X(\bfR^d)}/{\Diff(\bfR^d)}\xlra{\xi_X} {*}/{\Diff(\bfR^d)}\simeq \BO(d)}.\end{equation}
If $X\simeq E_M$ for a smooth $d$-manifold $M$, then $|X|\simeq M$ and $\xi_X$ models the tangent bundle $TM$ (see \cite[Proposition 5.10]{KKoperadic}). If $M$ is a \emph{closed} manifold, then the Pontryagin--Thom collapse map induced by the choice of an embedding $M\subset \bfR^{d+k}$ for $k\gg0$ gives a map $\bfS \to \Th(-TM)$ from the sphere spectrum to the Thom spectrum of the stable normal bundle of $M$, called the \emph{stable collapse map}, which is natural in diffeomorphisms. In \cref{sec:normal-invariants} we will show that it is even natural in equivalences of $2$-truncated presheaves, in the sense of the following theorem. In the statement and henceforth, we write $\IntMap(-,-)$ for mapping spectra and  $D(-) \coloneqq \IntMap(-,\bfS)$ for  Spanier--Whitehead duals.

\begin{thm}\label{thm:collapse-zigzag} To any $2$-truncated presheaf $X \in \PSh(\DiscInf_{d,\leq 2})$, one can associate a zig-zag of maps of spectra which is natural in equivalences in $\PSh(\DiscInf_{d,\leq 2})$,
	\begin{equation}\label{equ:collapse-zigzag}\bfS\ra D(\Sigma_+^\infty \lvert X\rvert )\leftarrow Z_X\ra \Th(-\xi_X)\quad\text{for some spectrum }Z_X,
	\end{equation}
	 and has the property that if $X=E_M$ for a closed smooth $d$-manifold $M$, then the wrong-way map is an equivalence and the map $\bfS\ra \Th(-TM)$ resulting from choosing an inverse is equivalent to the stable collapse map of $M$.
\end{thm}

\begin{rem}\label{rem:collapse}Some remarks on \cref{thm:collapse-zigzag}:
	\begin{enumerate}
		\item The construction in the proof \cref{thm:collapse-zigzag} is closely related to work of Naef--Safronov (c.f.\,\cite[Section 4.2]{NaefSafronov}, and \cref{rem:naefsafronov} below). There is  an alternative construction based on constructing collapse maps for $\DiscInf$-presheaves, inspired by \cite[Section 4]{KKsquare}.
		\item The first map in  \eqref{equ:collapse-zigzag} is induced by the unique map $\lvert X\rvert \ra \ast$, so considering only the right two maps in \eqref{equ:collapse-zigzag}, \cref{thm:collapse-zigzag} can be interpreted as saying that the Atiyah duality equivalence $D(\Sigma_+^\infty M) \simeq \Th(-TM)$ can be recovered from the $2$-truncated presheaf $E_M$ associated to $M$, naturally in equivalence of presheaves.
		\item A variant of \cref{thm:collapse-zigzag} holds in the setting of topological manifolds (see \cref{rem:topological-collapse}). 
		\item\label{enum:closed-presheaves} \cref{thm:collapse-zigzag} suggests the following notion: a $k$-truncated presheaf $X \in \PSh(\DiscInf_{d,\leq k})$ for $2\le k\le \infty$ is \emph{closed} if (a) $\lvert X\rvert$ lies in the full subcategory $\cS^\omega\subset \cS$ of compact objects in the $\infty$-category of spaces, (b) the wrong way map in \eqref{equ:collapse-zigzag} is an equivalence, (c) and the resulting map $c_X\colon \bfS\ra \Th(-\xi_X)$ exhibits $-\xi_X$ as the dualising spectrum of $\lvert X\rvert$ (in the sense of e.g.\ \cite[p.\,227]{LandReducibility}). In particular, in this case the underlying space $\lvert X\rvert$ is a $d$-dimensional Poincaré duality space (see p.\,228-220 loc.cit.).	
	\end{enumerate}
\end{rem}

Assuming \cref{thm:collapse-zigzag}, we prove \cref{thm:2-truncated-cokerJ} on $\DiscInf$-presheaves of $\coker(J)$-spheres:
\begin{proof}[Proof of \cref{thm:2-truncated-cokerJ}]
	We first recall the definition of the element $[\Sigma]\in\coker(J)_d$ associated to an oriented homotopy $d$-sphere (cf.\ \cite[Section 4]{KervaireMilnor}): Given an oriented homotopy $d$-sphere $\Sigma$, consider its stable collapse map $\bfS\ra\Th(-T\Sigma)$. Choosing a stable framing of $\Sigma$ gives an equivalence $\Th(-T\Sigma)\simeq \bfS^{-d}$, so the stable collapse map gives an element in $\pi_0(\bfS^{-d})\cong \pi_d(\bfS)$. Its image in $\coker(J)_d$ turns out to be independent of the choice of stable framing as long as the latter is chosen to be compatible with the given orientation of $\Sigma$; this defines the homomorphism $\Theta_d\ra\coker(J)_d$. In particular, we see that the class $[\Sigma]\in\coker(J)_d$ only depends on, (a) a classifier of the stable oriented tangent bundle $T\Sigma \colon \Sigma\ra \BSO$ and (b) the homotopy class of the stable collapse map $\bfS\ra \Th(-T\Sigma)$. By the discussion above, in particular \cref{thm:collapse-zigzag}, both of these can be recovered from the equivalence class of the $2$-truncated presheaf $E_\Sigma\in \PSh(\DiscInf^\plus_{d,\le 2})$, so the claim follows.
\end{proof}

\subsection{Digression: $\DiscInf$-structure spaces and $L$-theory}\label{sec:sdisc-ltheory}\cref{thm:collapse-zigzag} allows one to relate the $\DiscInf$-structure spaces from \cite{KKSDisc} to algebraic $L$-theory. We intend to take up this direction in future work, but briefly sketch how this goes:

We write $\smash{\Poinc^\simeq_d}$ for the $\infty$-groupoid of $d$-dimensional Poincaré spaces (in the sense of e.g.\,\cite[p.\,228-220]{LandReducibility}; this is a full subgroupoid of the core of the full subcategory $\cS^\omega$ of compact objects in the $\infty$-category $\cS$ of spaces) and $\Poinc^{\nu,\simeq}_d$ for the $\infty$-groupoid of $d$-dimensional Poincaré space together with a stable vector bundle refinement of its Spivak fibration (this can be constructed as a full subgroupoid of the core of the pullback of the composition $\smash{\cS^{\omega}_{/\bfZ\times \BO}\ra \cS^{\omega}_{/\Sp^\simeq}\ra\Sp}$ whose second functor is induced by taking colimits in the category $\Sp$ of spectra, along $\Sp_{\bfS/}\ra \Sp$).
In these terms \cref{thm:collapse-zigzag} yields a lift $\smash{\lvert-\rvert^\nu\colon\PSh^\cl(\DiscInf_{d})^\simeq\ra \Poinc^{\nu,\simeq}_d}$ of the functor $\smash{\lvert-\rvert\colon \PSh^\cl(\DiscInf_{d})^\simeq\ra \Poinc_d^\simeq}$, where $\smash{\PSh^\cl(\DiscInf_{d})\subset \PSh(\DiscInf_{d})}$ is the full subcategory of closed presheaves as in \cref{rem:collapse} \ref{enum:closed-presheaves}. Moreover, with some effort, one ought to be able to construct a commutative square 
\begin{equation}
\begin{tikzcd}[row sep=0.5cm]
\ManInf^{\cl,\cong}_d\rar{E}\arrow[d,"\forget",swap]&\PSh^\cl(\DiscInf_{d})^\simeq\dar{\lvert-\rvert^\nu}\\
\widetilde{\ManInf}^{\cl,\cong}_d\rar{\nu}&\Poinc^{\nu,\simeq}_d
\end{tikzcd}
\end{equation}
where $\smash{\ManInf^{\cl,\cong}_d}$ (respectively $\smash{\widetilde{\ManInf}^{\cl,\cong}_d}$) is the $\infty$-groupoid of closed $d$-dimensional smooth manifolds and spaces of diffeomorphisms (respectively block-diffeomorphisms) between them. The map $\nu$ is induced by taking stable normal bundles and constructed such that it agrees after taking fibres over $\Poinc_d^\simeq$ with the normal invariant map in surgery theory. The fibre of the top map at $E_M$ for $\smash{M\in \ManInf^{\cl,\cong}_d}$ is by definition the $\DiscInf$-structure space $\smash{S^{\DiscInf}(M)}$ as introduced in \cite{KKSDisc}, so taking horizontal fibres yields a map \[\smash{S^{\DiscInf}(M)\lra \fib_{\nu(M)}(\widetilde{\ManInf}^{\cl,\cong}_d\ra \Poinc^{\nu,\simeq}_d).}\]
Moreover, if $d\ge 5$, then by surgery theory, the collection of the components $(N, \nu(M)\simeq \nu(N))$ of the target whose underlying homotopy equivalence $M\simeq N$ is simple (with the finiteness structure on $M$ and $N$ induced from their manifold structures), is equivalent to loop space of the quadratic $L$-theory space $L^q(M)$ of $M$, so writing $S^{\DiscInf}(M)^s\subset S^{\DiscInf}(M)$ for the components of $(N,E_M\simeq E_N)$ whose underlying homotopy equivalence $M\simeq N$ is simple (note that we have $S^{\DiscInf}(M)=S^{\DiscInf}(M)^s$ if $M$ is simply connected, or more generally---by \cite[Corollary C]{NaefSafronov}---if the Dennis trace $\Wh(\pi_1(M))\ra H_1(LM,M)$ is injective), one obtains a map of the form
\begin{equation}\label{equ:map-to-l}
\smash{S^{\DiscInf}(M)^s\lra \Omega L^q(M).}
\end{equation}
As a result of \cite[Theorem A]{KKSDisc}, the source of this map depends up to equivalence only on the tangential $2$-type of $M$, and by the $\pi\text{-}\pi$-theorem the same holds for the target (in fact it only depends on the $1$-truncation of $M$ and the first Stiefel--Whitney class). It seems conceivable that the map \eqref{equ:map-to-l} also only depends on the tangential $2$-type of $M$.

\section{Gluing $\DiscInf$-presheaves}\label{sec:gluing}
This section serves to establish the first of the three ingredients assumed in \cref{sec:main-ideas}: \cref{thm:gluing-isotopy} on the dependence of $\DiscInf$-presheaves of glued manifolds on the gluing diffeomorphism. We will deduce this from a general gluing-result for presheaves.

\begin{nrem}Throughout this section, we assume familiarity with $\infty$-analogues of standard concepts in category theory, as developed in \cite{LurieHTT, LurieHA}. In particular Kan extensions (see \cite[Section 4.3]{LurieHTT}) and Beck--Chevalley conditions (also called \emph{adjointability conditions}; see \cite[4.7.4.13-4.7.4.15]{LurieHA}) will play a key role. 
\end{nrem}

\subsection{Gluing $\DiscInf$-presheaves}
\subsubsection{Manifolds with boundary}\label{sec:manpartial}
Write $\smash{\ManInf^\partial_d}$ for the $\infty$-category whose objects are $d$-manifolds $M$ (potentially non-compact and with boundary) and whose morphisms are spaces $\Emb^\partial(M,N)$ of embeddings $e\colon M\hookrightarrow N$ with $e^{-1}(\partial N)=\partial M$ (see \cref{sec:point-set-models-mfd-cats} for a precise construction of $\smash{\ManInf^\partial_d}$). It contains as a full subcategory the category $\ManInf_d$ of $d$-manifolds without boundary and embeddings between them. Taking boundaries yields a functor $\smash{\partial\colon \ManInf^\partial _{d}\ra \ManInf_{d-1}}$ which has a fully faithful left adjoint $\smash{\kappa\colon  \ManInf_{d-1}\ra \ManInf^\partial _{d}}$ given by taking products with $[0,\infty)$, i.e.\,it sends $P\in \ManInf_{d-1}$ to $P\times[0,\infty)$. The reason that these functors are adjoint is that the restriction map $\Emb^\partial (P\times[0,\infty),M)\ra \Emb(P,\partial M)$ is an equivalence for all $P\in\ManInf_{d-1}$ and $M\in\ManInf^\partial_d$, as its fibres are equivalent to spaces of collars, which are contractible. Note that the unit of this adjunction $\id\ra \partial \kappa$ is an equivalence, as $P=\partial(P\times[0,\infty))$. Since $\partial$ is right adjoint to $\kappa$, the left Kan extension $\smash{\partial_!\colon \PSh(\ManInf^\partial_d)\ra \PSh(\ManInf^\partial_{d-1})}$ is right adjoint to $\kappa_!$ (the fact that passing to left Kan extensions preserves adjoints follows by observing that it preserves the triangle identities that characterise adjunctions), so since restriction is also left adjoint to left Kan extension, we get an equivalence $\partial_!\simeq \kappa^*$ of functors $\PSh(\ManInf^\partial_d)\ra \PSh(\ManInf_{d-1})$. 

\subsubsection{$\DiscInf$-subcategories}Writing $\bfH^d\coloneq [0,\infty)\times \bfR^{d-1}$ for the $d$-dimensional halfspace, we consider the nested sequence of full subcategories of $\ManInf^\partial_d$
\begin{equation}\label{equ:cotower-disc-partial}\smash{\DiscInf_{d,\le 1}^\partial\subset \DiscInf_{d,\le 2}^\partial\subset \cdots \subset \DiscInf_{d,\le \infty}^\partial=\DiscInf_{d}^\partial}\end{equation}
 where $\smash{\DiscInf_{d,\le k}^\partial\subset \ManInf^\partial_d}$ is the full subcategory on those manifolds that are diffeomorphic to $S\times \bfR^d\sqcup T\times \bfH^{d-1}$ for finite sets $S$ and $T$ with $\lvert S\rvert+\lvert T\rvert \le k$.  Note that intersecting \eqref{equ:cotower-disc-partial} with $\smash{\ManInf_d\subset \ManInf^\partial_d}$ yields the nested sequence of full subcategories $\smash{\DiscInf_{d,\le\bullet}\subset \ManInf_d}$ on manifolds diffeomorphic to $S\times \bfR^d$ with $|S|\le \bullet$, as in the introduction. The functor $\kappa$ preserves these subcategories and gives a map of nested sequences $\kappa\colon  \DiscInf_{d-1,\le\bullet}\ra \DiscInf^\partial_{d,\le\bullet}$, and thus a map of towers (in the sense of \cite[Section 1.2]{KKoperadic}) of categories $\smash{\kappa^*\colon \PSh(\DiscInf^\partial_{d,\le\bullet})\ra \PSh(\DiscInf_{d-1,\le\bullet})}$ by restriction. The adjunction $\kappa\vdash \partial$ from \cref{sec:manpartial} restricts to adjunctions of the form $\smash{\kappa\colon \DiscInf_{d-1,\le k} \rightleftarrows\DiscInf^\partial_{d,\le k}\colon  \partial}$ for all $k$, so as above we obtain $\partial_!\simeq \kappa^*$ as functors $\PSh(\DiscInf^\partial_{d,\le k})\ra \PSh(\DiscInf_{d-1,\le k})$. Writing $\smash{\iota^\partial \colon \DiscInf_{d,\le \bullet}^\partial\hookrightarrow \ManInf^\partial_d}$ and $\smash{\iota \colon \DiscInf_{d-1,\le \bullet}\hookrightarrow \ManInf_{d-1}}$ for the respective inclusions, this implies that the square 
\begin{equation}\label{equ:bc-mandiscpartial}
\begin{tikzcd}[row sep=0.5cm]
\PSh(\ManInf_d^\partial)\rar{(\iota^\partial)^*}\arrow["\partial_!",swap,d]&\PSh(\DiscInf_{d,\le \bullet}^\partial)\dar{\partial_!}\\
\PSh(\ManInf_{d-1})\rar{\iota^*}&\PSh(\DiscInf_{d-1,\le \bullet})
\end{tikzcd}
\end{equation}
of towers of categories commutes in that the Beck--Chevalley transformation $\partial_! (\iota^\partial)^*\ra \iota^*\partial_!$ is an equivalence. We denote by \[E^\partial\coloneqq((\iota^\partial)^*\circ y)\colon \ManInf^\partial_d\lra \PSh(\DiscInf^\partial_{d,\le\bullet})\] the tower of restricted Yoneda embeddings.

\subsubsection{Gluing manifolds and $\DiscInf$-presheaves}\label{sec:glue-disc-ps}
Writing
\[\smash{\ManInf_d^{\midman}\coloneq \ManInf_d^\partial\times_{\ManInf_{d-1}}\ManInf_d^\partial}\] for the pullback of $\partial\colon \ManInf_d^\partial\ra \ManInf_{d-1} $ along itself, naturality of the Yoneda embedding yields a functor \begin{equation}\label{equ:glued-mfd-yoneda}\ManInf_d^{\midman}\ra \PSh(\ManInf_d^\partial)\times_{\PSh(\ManInf_{d-1})}\PSh(\ManInf_d^\partial)\end{equation} to the pullback of the left vertical functor in \eqref{equ:bc-mandiscpartial} along itself. 

\begin{rem}\label{rem:objects-in-manmid} An object in $\ManInf_d^{\midman}$ is given by a triple $(M_0,M_1,\varphi)$ where the $M_i$ are $d$-manifolds and $\varphi\colon \partial M_0\ra\partial M_1$ is a diffeomorphism between their boundaries. The spaces of morphisms in $\ManInf_d^{\midman}$ is given by the pullback of embedding spaces $\Emb^\partial(M_0,N_0)\times_{\Emb(\partial M_0,\partial N_1)}\Emb^{\partial}(M_1,N_1)$. Note that given a $d$-manifold $W$ without boundary with a decomposition $W=W_0\cup W_1$ into two codimension $0$ submanifolds that intersect in their common boundary $P\coloneq \partial_0W=\partial_1W$, we obtain an object $(W_0,W_1,\id_P)$, and it turns out that any object $(M_0,M_1,\varphi)$ is equivalent to one of this form, by considering the glued manifold $W=M_0\cup_{\varphi} M_1$.
\end{rem}

We now consider the functor 
\begin{equation}\label{equ:e-midman}E^{\midman}\colon \ManInf_d^{\midman}\lra \PSh(\DiscInf_{d,\le\bullet}^\partial)\times_{\PSh(\DiscInf_{d-1,\le\bullet})}\PSh(\DiscInf_{d,\le\bullet}^\partial).
\end{equation}
given by the post composition of \eqref{equ:glued-mfd-yoneda} with the functor $\PSh(\ManInf_d^\partial)\times_{\PSh(\ManInf_{d-1})}\PSh(\ManInf_d^\partial)\ra \PSh(\DiscInf_{d,\le\bullet}^\partial)\times_{\PSh(\DiscInf_{d-1,\le\bullet})}\PSh(\DiscInf_{d,\le\bullet}^\partial)$ induced by the commutative square \eqref{equ:bc-mandiscpartial}. In arity $1\le k\le \infty$, the functor \eqref{equ:e-midman} sends a triple $(M_0,M_1,\varphi)$ as in \cref{rem:objects-in-manmid} to the triple consisting of the presheaves $E^\partial_{M}$ and $E^\partial_{N})$ in $\smash{\PSh(\DiscInf^{\partial}_{d,\le k})}$, together with the equivalence $\partial_!(E^\partial_{M})\simeq E_{\partial M}\simeq E_{\partial N}\simeq \partial_!(E^\partial_{N})$ in $\smash{\PSh(\DiscInf_{d-1,\le k})}$ induced by the diffeomorphism $\partial M\cong \partial N$ and \eqref{equ:bc-mandiscpartial}. The goal of the remainder of this section is to prove that the latter data $(E^\partial_{M},E^\partial_{N},E_{\partial M}\simeq E_{\partial N})$ is sufficient to reconstruct the presheaf $E_{M\cup_\partial N}\in\PSh(\DiscInf_{d,\le k})$ of the glued manifold. More concretely, gluing manifolds along their boundary yields a functor $\smash{\gamma\colon \ManInf_d^{\midman}\ra \ManInf}$ (see \cref{sec:point-set-models-mfd-cats}) which features in the following theorem:

\begin{thm}\label{thm:gluing}There exists a commutative square of towers of categories
\[
\begin{tikzcd}[row sep=0.5cm]
\ManInf_d^\midman\arrow[d,"\gamma",swap]\rar{E^{\midman}}&\PSh(\DiscInf_{d,\le\bullet}^\partial)\times_{\PSh(\DiscInf_{d-1,\le\bullet})}\PSh(\DiscInf_{d,\le\bullet}^\partial)\dar{\gamma^{\DiscInf}}\\
\ManInf_{d}\rar{E}&\PSh(\DiscInf_{d,\le\bullet})
\end{tikzcd}
\]
for some functor of towers $\smash{\gamma^{\DiscInf}}$ with the indicated source and target.
\end{thm}

\begin{rem}[Variants of \cref{thm:gluing}]\label{rem:extensions-gluing}\ 
\begin{enumerate}
\item\label{enum:extensions-gluing-i} There is a variant of \cref{thm:gluing} for topological manifolds, by replacing all categories of (smooth) manifolds and (smooth) embeddings between them by the corresponding categories of topological manifolds and topological embeddings between them. The proof is the same as in the smooth setting.
\item \label{enum:extensions-gluing-ii}There is also a variant of \cref{thm:gluing} for manifolds with tangential structures with respect to a map $\theta\colon B\ra \BO(d)$ from some space $B$: Given a map $\theta\colon B\ra \BO(d)$ from a space $B$, there are variants $\ManInf_d^{\theta,\partial}$ and $\ManInf_{d-1}^{\theta}$ of $\ManInf_d^{\partial}$ and $\ManInf_{d-1}$ where the manifolds $M$ of dimension $d$ or $d-1$ are additionally equipped with a \emph{$\theta$-structure}, i.e.\,a bundle map $TM\ra \theta^*\gamma_d$ in the $d$-dimensional case or $TM\oplus\varepsilon \ra \theta^*\gamma_d$ in the $(d-1)$-dimensional case; here $\gamma_d$ denotes the universal $d$-dimensional bundle. The functor $\partial$ has {two} lifts to a functor $\smash{\partial\colon \ManInf_d^{\theta,\partial}\ra  \ManInf_{d-1}^{\partial}}$, induced by restricting a $\theta$-structure along the inclusion $T(\partial M)\oplus\varepsilon\ra TM$ using either an inwards or an outwards pointing vector field. Forming the pullback $\smash{\ManInf_d^{\theta,\midman}\coloneqq\ManInf^{\theta,\partial}_d\times_{\ManInf^{\theta}_{d-1}}\ManInf^{\theta,\partial}_d}$ with respect to the two different lifts of $\partial$, there is a lift of the gluing functor $\gamma$ to a functor $\ManInf^{\theta,\midman}_d\ra \ManInf^\theta_d$, and this functor fits into diagram as in \cref{thm:gluing} if one replaces $\smash{\DiscInf^{\partial}_{d,\le\bullet}}$ and $\smash{\DiscInf_{d-1,\le\bullet}}$ by $\smash{\DiscInf^{\theta,\partial}_{d,\le\bullet}\coloneq\DiscInf^{\partial}_{d,\le\bullet}\times_{\DiscInf^{\partial}_{d}}\ManInf_d^{\theta,\partial}}$ and $\smash{\DiscInf^{\theta}_{d-1,\le\bullet}\coloneq\DiscInf_{d-1,\le\bullet}\times_{\DiscInf_{d-1}}\ManInf_{d-1}^{\theta}}$. The topological version alluded to in \ref{enum:extensions-gluing-i} can be generalised analogously to include $\theta$-structures with respect to a map $\theta\colon B\ra\BTop(d)$. These versions with tangential structure can be established by adapting the proof of \cref{thm:gluing}.
\end{enumerate}
\end{rem}
\begin{rem}There is a similar commutative square of towers
\begin{equation}\label{equ:square-alg-mod}\hspace{0.6cm}
\begin{tikzcd}[row sep=0.5cm, column sep=0.3cm]
\ManInf_d^\midman\arrow[d,"\gamma",swap]\rar&\RMod(\PSh(\DiscInf_{d,\le\bullet}))\times_{\AAlg(\PSh(\DiscInf_{d,\le\bullet}))}\LMod(\PSh(\DiscInf_{d,\le\bullet}))\dar{\otimes}\\
\ManInf_{d}\rar{E}&\PSh(\DiscInf_{d,\le\bullet})
\end{tikzcd}
\end{equation}
where (a) $\LMod(-)$, $\RMod(-)$, $\AAlg(-)$ are the categories of pairs of an associative algebra and a left-module over it, pairs of an associative algebra and a right-module over it, and associative algebras, in a symmetric monoidal category respectively, (b) the symmetric monoidal structure on $\PSh(\DiscInf_{d,\le k})$ is a localisation of the Day convolution structure induced from the symmetric monoidal structure on $\DiscInf_d$ by disjoint union, and (c) $\otimes$ denotes the relative tensor product of modules. This square can be extracted from \cite[Theorem 5.3]{KKoperadic} (in the nontruncated case $k=\infty$, it is also an instance of what is known as \emph{$\otimes$-excision} in the theory of factorisation homology; see \cite[Proposition 3.24]{Francis} and \cite[Lemma 3.18]{AyalaFrancis}). It implies a version of \cref{thm:gluing-isotopy} where the condition \[\smash{[E_{\varphi_0}]=[E_{\varphi_1}]\in\pi_0\,\Map_{\PSh(\DiscInf_{d-1,\le k})}(E_{\partial M_0},E_{\partial M_1}),}\]
is replaced by 
\[\smash{[E_{\varphi_0\times\bfR}]=[E_{\varphi_1\times\bfR}]\in\pi_0\,\Map_{\AAlg(\PSh(\DiscInf_{d-1,\le k}))}(E_{\partial M_0\times\bfR},E_{\partial M_1\times\bfR}).}\]
However, this is not sufficient for the proof of \cref{thm:bpspheres}, since the result from \cite{Mezher} the proof relies on---namely that $\pi_0\,\Aut_{\PSh(\DiscInf_{d-1,\le k})}(E_M)$ is residually finite for $k<\infty$ and many manifolds $M$---is not available for $\pi_0\,\Aut_{\AAlg(\PSh(\DiscInf_{d-1,\le k}))}(E_{M\times\bfR})$. To circumvent this, it would be sufficient to know that the map of towers $\ManInf_{d-1}\ra \AAlg(\PSh(\DiscInf_{d,\le\bullet}))$ participating in \eqref{equ:square-alg-mod} factors over the map of towers $E\colon \ManInf_{d-1}\ra \PSh(\DiscInf_{d-1,\le\bullet})$. This can be shown to be indeed the case, but it is not immediate. We decided against pursuing this route in this work when we discovered the construction of the square in \cref{thm:gluing} which does not rely on $\otimes$-excision, avoids algebras, modules, and tensor products, and thus might be of independent interest.
\end{rem}

\begin{rem}[Point-set models]\label{sec:point-set-models-mfd-cats}Above we only gave informal definitions of the categories $\smash{\ManInf_d^{\partial}}$ and $\ManInf_{d-1}$ in that we only described their objects and spaces of morphisms, and we also only informally specified the functors $\smash{\partial\colon \ManInf_d^{\partial}\ra \ManInf_{d-1}}$ and $\smash{\gamma\colon  \ManInf_d^{\partial}\times_{ \ManInf_{d-1}}\ManInf_d^{\partial}\ra \ManInf_d}$ (all other constructions, however, were formally obtained from these). There are many options for how to construct these categories and functors, one is by defining them explicitly on the level of Kan-enriched categories and taking coheren nerves, another is to extract them from the constructions that were already carried out in   \cite[Section 3]{KKSDisc}. We explain the latter:

Recall (e.g.\,from \cite[Section 1.1]{KKoperadic}) that a \emph{double category} $\cM$ is a category-object in $\CatInf$, i.e.\,a simplicial object $\smash{\cM\in\Fun(\Delta^\op,\CatInf)}$ in categories that satisfies the Segal condition. It has a \emph{category of objects} $\ob(\cM)\coloneqq \cM_{[0]}$ and for $c,d\in\ob(\cM)$ a \emph{category of morphisms} $\cM_{c,d}\coloneqq \{c\}\times_{\cM_{[0]} }\cM_{[1]}\times_{\cM_{[0]}} \{d\}$ where the pullback is taken over the \emph{source and target functors} induced by $0,1\colon [0]\ra [1]$. Letting  source or target vary, we write $\cM_{c,-}\coloneqq \{c\}\times_{\cM_{[0]} }\cM_{[1]}$ and $\cM_{-,d}\colon \cM_{[1]}\times_{\cM_{[0]}} \{d\}$, and these come with functors $t\colon \cM_{c,-}\ra \ob(\cM)$ and  $s\colon \cM_{-,d}\ra \ob(\cM)$ by taking targets or sources, respectively. We have a \emph{composition functor} \[\smash{\cM_{c,-}\times_{\ob(\cM)}\cM_{-,d}   \simeq \{c\}\times_{\cM_{[0]} }\cM_{[2]}\times_{\cM_{[0]} }\{d\}\xlra{(0\le 2)^*} \{c\}\times_{\cM_{[0]} }\cM_{[1]}\times_{\cM_{[0]} }\{d\}= \cM_{c,d}}\] where the first equivalence uses the Segal condition. The \emph{opposite} $\cM^\op$ of a double category $\cM$ is the double category obtained by precomposing the simplicial object $\cM$ with the functor $\op\colon \Delta\ra \Delta$ with $\op([n])=[n]$ and $\op(\alpha\colon [m]\ra [n])(i)=n-\alpha(m-i))$. Note that one has $\ob(\cM)=\ob(\cM^{\op})$ and $\cM^{\op}_{c,d}=\cM_{d,c}$.

In \cite[Section 3, Steps \circled{1} and \circled{5}]{KKSDisc}, we constructed a double category $\ncBordInf_d$ of (possibly noncompact) $(d-1)$-manifolds and bordisms between them. It comes with an anti-involution $(-)^\rev\colon \ncBordInf_d\ra (\ncBordInf_d)^\op$ given by ``reversing bordisms'', which is---in the language of Section 3 Steps \circled{1} loc.cit.---induced by sending a $[p]$-walled $d$-manifold $(W\subset \bfR\times \bfR^\infty, \mu\colon [p]\ra \bfR)$ to $((-1\times_{\id_{\bfR^\infty}})(W),((-1)\circ \mu\circ (i\mapsto p-i))\colon [p]\ra \bfR)$. Moreover, there is an equivalence $(\ncBordInf_{d})_{\varnothing,\varnothing}\simeq \ob(\ncBordInf_{d+1})$ induced by sending a $[1]$-walled $d$-manifold $(W\subset \bfR\times \bfR^\infty, \mu\colon [1]\ra \bfR)$ to $(\bfR\times W|_{[\mu(0),\mu(1)]}, 0\colon [0]\ra \bfR)$. We set \[\ManInf_{d-1}\coloneq \ob(\ncBordInf_d)\quad\text{and}\quad\ManInf_{d}^\partial\coloneq (\ncBordInf_d)_{\varnothing,-}.\] The functor $\partial\colon \ManInf_{d}^\partial\ra \ManInf_{d-1}$ is defined as $t\colon (\ncBordInf_d)_{\varnothing,-}\ra \ob(\ncBordInf_d)$ and the functor $\smash{\gamma\colon \ManInf_d^\midman=\ManInf_d^{\partial}\times_{ \ManInf_{d-1}}\ManInf_d^{\partial}\ra \ManInf_d}$ is given by the composition of the equivalence $\smash{(\ncBordInf_d)_{\varnothing,-}\times_{\ob(\ncBordInf_d)}(\ncBordInf_d)_{\varnothing,-}\simeq (\ncBordInf_d)_{\varnothing,-}\times_{\ob(\ncBordInf_d)}(\ncBordInf_d)_{-,\varnothing}}$ which is induced by the anti-involution $(-)^\rev$ in the second argument, with the composition functor $\smash{(\ncBordInf_d)_{\varnothing,-}\times_{\ob(\ncBordInf_d)}(\ncBordInf_d)_{-,\varnothing}\ra (\ncBordInf_d)_{\varnothing,\varnothing}\simeq \ob(\ncBordInf_{d+1})=\ManInf_d}$.
\end{rem}

\subsection{Some category theory}The proof of \cref{thm:gluing} relies on the following lemma:
\begin{lem}\label{lem:pushout-lem}Fix categories $\cD_{01},\cD_0,\cD_1$ and fully faithful functors $\kappa_i\colon \cD_{01}\hookrightarrow \cD_i$ for $i=0,1$ that admit right adjoints $\partial_i\colon \cD_i\ra \cD_{01}$.
\begin{enumerate}
\item\label{enum:pushout-ff} The natural functor $\iota\colon \cD_0\cup_{\cD_{01}}\cD_1\ra \cD_0\times_{\cD_{01}}\cD_1$ from the pushout of the $\kappa_i$ to the pullback of the $\partial_i$, is fully faithful.
\item\label{enum:right-Kan}  The value at an object $(d_0,d_1)\in \cD_0\times_{\cD_{01}}\cD_1$ of the unit $X\ra \iota_*\iota^*X$ of the adjunction $\iota^*\colon\PSh(\cD_0\times_{\cD_{01}}\cD_1)\rightleftarrows\PSh(\cD_0\cup_{\cD_{01}}\cD_1)\colon \iota_*$ between restriction and right Kan extension is naturally equivalent to the map from the top left-corner in the square\vspace{-0.1cm}
\[
\begin{tikzcd}[row sep=0.5cm]
X(d_0,d_1)\rar{(\id,\epsilon_1)^*}\arrow["{(\epsilon_0,\id)^*}",d,swap]&X(d_0,\kappa_1\partial_1(d_1))\dar{(\epsilon_0,\id)^*}\\
X(\kappa_0\partial_0(d_0),d_1)\rar{(\id,\epsilon_1)^*}&X(\kappa_0\partial_0(d_0),\kappa_1\partial_1(d_1))
\end{tikzcd}
\]
to the pullback of the remaining entries.  Here the maps $\epsilon_i$ are the counits of $\kappa_i\vdash \partial_i$.
\item\label{enum:BC} For full subcategories $\cD_{01}'\subset \cD_{01}$ and $\cD_{i}'\subset \cD_{i}$ for $i=0,1$ to which the functors $\kappa_i$ and $\partial_i$ restrict, the diagram of categories of presheaves
\[
\begin{tikzcd}[row sep=0.5cm]
\PSh(\cD_0\cup_{\cD_{01}}\cD_1)\arrow[d,swap,"\iota_*"]\rar{\inc^*}&\PSh(\cD'_0\cup_{\cD'_{01}}\cD'_1)\dar{\iota'_*}\\
\PSh(\cD_0\times_{\cD_{01}}\cD_1)\rar{\inc^*}&\PSh(\cD'_0\times_{\cD'_{01}}\cD'_1)
\end{tikzcd}
\]
commutes, i.e.\,the Beck--Chevalley transformation $\inc^*\iota_*\ra \iota_*'\inc^*$ is an equivalence.
\item\label{enum:restricted-yoneda} In the situation of \ref{enum:BC}, the unit $\id\ra \iota'_*(\iota')^*$ is an equivalence on the essential image of the restricted Yoneda embedding $(\inc^* \circ y)\colon \cD_{0}\times_{\cD_{01}}\cD_1\ra \PSh(\cD'_{0}\times_{\cD'_{01}}\cD'_1)$.
\end{enumerate}
\end{lem}

\begin{proof}
We begin with a few preliminary observations:
\begin{enumerate}[leftmargin=*,label=(\alph*)]
\item\label{enum:obs-i} Using the adjunctions   $\kappa_i\vdash\partial_i$ and that mapping spaces in pullbacks of categories are given by the pullbacks of mapping spaces, one sees that the two projections $\cD_0\times_{\cD_{01}}\cD_1\ra\cD_i$ 
induce for $d_i,d_i'\in\cD_i$ equivalences  \vspace{-0.1cm}\begin{align*}\Map_{\cD_0\times_{\cD_{01}}\cD_1}((d_0,\kappa_1\partial_0(d_0)),(d_0',d_1'))&\overset{\simeq}\lra \Map_{\cD_0}(d_0,d_0'), \\[-0.1cm] \Map_{\cD_0\times_{\cD_{01}}\cD_1}((\kappa_0\partial_1(d_1),d_1),(d_0',d_1'))&\overset{\simeq}\lra \Map_{\cD_1}(d_1,d_1').\end{align*}
\item \label{enum:obs-ii}The functor $\iota$ in \ref{enum:pushout-ff} is induced by the two functors $(\id,\kappa_1\partial_0)\colon \cD_0\ra \cD_0\times_{\cD_{01}}\cD_1$ and $(\kappa_0\partial_1,\id)\colon \cD_1\ra \cD_0\times_{\cD_{01}}\cD_1$ which are both fully faithful as a result of \ref{enum:obs-i}, and agree on $\cD_{01}$ since $\partial_1 \kappa_1 \simeq \id \simeq \partial_0\kappa_0$ as the $\kappa_i$ were assumed to be fully faithful.
\item\label{enum:obs-iii} Since fully faithful functors are preserved by taking pushouts \cite[Theorem 0.1]{HRSpushouts}, the inclusion functors $\cD_i\ra \cD_0\cup_{\cD_{01}}\cD_1$ are fully faithful for $i=0,1$.
\end{enumerate}
Combining \ref{enum:obs-ii} with \ref{enum:obs-iii} and observing that $\cD_i \to \cD_0 \cup_{\cD_{01}} \cD_1$ for $i=0,1$ are jointly essentially surjective, to show the claim in \ref{enum:pushout-ff} it suffices to show that for $d_i\in\cD_i$, the functor induces an equivalence on the mapping space from $d_0$ to $d_1$ and on that from $d_1$ to $d_0$. By symmetry, it suffices to check the former. By Theorem 0.1 loc.cit., the map 
\begin{equation}\label{equ:mapping-space-pushout}
\begin{tikzcd}[row sep=0cm, column sep=0.3cm]
\lvert (\cD_0)_{d_0/}\times_{\cD_{0}}\cD_{01}\times_{\cD_1}(\cD_1)_{/d_1}\rvert\rar& \lvert 
(\cD_0\cup_{\cD_{01}}\cD_1)_{d_0/}\times_{(\cD_0\cup_{\cD_{01}}\cD_1)}(\cD_0\cup_{\cD_{01}}\cD_1)_{/d_1}\rvert \\
&\simeq \Map_{\cD_0\cup_{\cD_{01}}\cD_1}(d_0,d_1)
\end{tikzcd}
\end{equation}
induced by the inclusions of $\cD_0,\cD_1,\cD_{01}$ into $\cD_0\cup_{\cD_{01}}\cD_1$ is an equivalence (for the equivalence in \eqref{equ:mapping-space-pushout}, see Remark 3.4 loc.cit.). Here $\lvert-\rvert\colon \CatInf\ra\cS$ denotes the left adjoint to the inclusion $\cS\subset \CatInf$, and $(-)_{(-)/}$ and $(-)_{/(-)}$ indicate under- and overcategories. Now consider the commutative diagram of pullbacks of categories
\begin{equation}\label{equ:pullback-diagram}
\begin{tikzcd}[row sep=0.5cm]
\Map_{\cD_0}(d_0,\kappa_0\partial_1(d_1))\rar\dar&(\cD_0)_{d_0/}\times_{\cD_{0}}\big(\cD_{01}\times_{\cD_1}(\cD_1)_{/d_1}\big)\rar{\pr_1}\dar{\pr_2} &(\cD_{0})_{d_0/}\dar{\forget}\\
*\rar{(\partial_1(d_1),\kappa_1\partial_1(d_1)\xra{\epsilon_1} d_1)}&\cD_{01}\times_{\cD_1}(\cD_1)_{/d_1}\rar{\kappa_0\pr_1}& \cD_0.
\end{tikzcd}
\end{equation}
We claim that the bottom left horizontal map is the inclusion of a terminal object. To show this, we fix an object $(c_{01},f\colon \kappa_1(c_{01})\ra d_1)\in \cD_{01}\times_{\cD_1}(\cD_1)_{/d_1}$ and consider the commutative diagram of mapping spaces
\[\hspace{-0.1cm}
\begin{tikzcd}[column sep=0.2cm]
\Map_{\cD_{01}\times_{\cD_1}(\cD_1)_{/d_1}}\big((c_{01},f),(\partial_1(d_1),\epsilon_1)\big)\arrow[r,"\pr_2"]\arrow[d,"\pr_1"]&\Map_{(\cD_1)/_{d_1}}(f,\epsilon_1)\arrow[r]\arrow[d,"\forget"]&*\arrow[d,"f"]\\
\Map_{\cD_{01}}(c_{01},\partial_1(d_1))\rar{\kappa_1}&\Map_{\cD_1}(\kappa_1(c_{01}),\kappa_1(\partial_1(d_1)))\rar{(\epsilon_1)_*}&\Map_{\cD_1}(\kappa_1(c_{01}),d_1).
\end{tikzcd}
\]
Both squares are pullbacks, by the description of mapping spaces in pullbacks and overcategories as pullbacks and fibres, respectively. Hence the outer square is also a pullback. But the bottom composition is an equivalence by the adjunction $\kappa_1\vdash\partial_1$, hence the top composition is an equivalence as well, so the top-left corner is contractible which confirms that $(\partial_1(d_1),\epsilon_1)$ is a terminal object. Now note that the rightmost vertical map in \eqref{equ:pullback-diagram} is a cocartesian fibration (it classifies the functor $\Map_{\cD_0}(d_0,-)$), so as cocartesian fibrations are closed under pullback \cite[2.4.2.3]{LurieHTT}, $\pr_2$ in \eqref{equ:pullback-diagram} is a cocartesian fibration as well and is thus by 4.1.2.15 loc.cit.\,smooth in the sense of 4.1.2.9 loc.cit.. Since we saw that the bottom left horizontal map is the inclusion of a terminal object and therefore cofinal, it follows from 4.1.2.10 loc.cit.\ that the upper left horizontal map in \eqref{equ:pullback-diagram} is also cofinal and thus an equivalence after applying $\lvert-\rvert$ by 4.1.1.3 (3) loc.cit.. Combining this with \eqref{equ:mapping-space-pushout}, we conclude that the composition
\[
\begin{tikzcd}[row sep=0.3cm, row sep=0.1cm]\Map_{\cD_0}(d_0,\kappa_0\partial_1(d_1))\rar&\Map_{\cD_0\cup_{\cD_{01}}\cD_1}(d_0,\kappa_0\partial_1(d_1))\arrow[d,phantom,"\simeq"] \\
&  \Map_{\cD_0\cup_{\cD_{01}}\cD_1}(d_0,\kappa_1\partial_1(d_1))\rar&\Map_{\cD_0\cup_{\cD_{01}}\cD_1}(d_0,d_1)
\end{tikzcd}
\]
induced by the inclusion $\cD_0\ra \cD_0\cup_{\cD_{01}}\cD_1$ and postcompostion with the counit $\kappa_1\partial_1(d_1)\ra d_1$, is an equivalence. This implies the claim in \ref{enum:pushout-ff}, since postcomposition of this composition with the map induced by $\iota$ is equivalent to the identity on ${\Map_{\cD_0}(d_0,\kappa_0\partial_1(d_1))}$ as a result of \ref{enum:obs-ii}.

We now turn to proving \ref{enum:right-Kan}. By the limit-formula for right Kan extensions \cite[4.3.2.13]{LurieHTT}, the unit map in the claim is naturally equivalent to the map $X(d_0,d_1)\ra \lim_{(\iota(d)\ra (d_0,d_1))}X(\iota(d))$ where the limit is over $\smash{(\cD_0\cup_{\cD_{01}}\cD_1)_{/(d_0,d_1)}\coloneq (\cD_0\cup_{\cD_{01}}\cD_1)\times_{(\cD_0\times_{\cD_{01}}\cD_1)}(\cD_0\times_{\cD_{01}}\cD_1)_{/(d_0,d_1)}}$. The latter category is the value of at $\iota\colon \cD_0\cup_{\cD_{01}}\cD_1\ra \cD_0\times_{\cD_{01}}\cD_1$ of the functor 
\begin{equation}\label{equ:pullback-functor}(-)_{/(d_0,d_1)}\colon (\CatInf)_{/\cD_0\times_{\cD_{01}}\cD_1}\ra (\CatInf)_{/((\cD_0\times_{\cD_{01}}\cD_1)_{/(d_0,d_1)})}\end{equation}
which is given by taking pullback along the forgetful functor $\pi\colon (\cD_0\times_{\cD_{01}}\cD_1)_{/(d_0,d_1)}\ra\cD_0\times_{\cD_{01}}\cD_1$. Since the latter forgetful functor is a cartesian fibration (it classifies the functor $\smash{\Map_{\cD_0\times_{\cD_{01}}\cD_1}(-,(d_0,d_1))}$), the pullback functor \eqref{equ:pullback-functor} is a left adjoint by \cite[Definition 1.1.6 and Lemma 3.2.1]{AyalaFrancisFibrations}, so it preserves colimits. As colimits in overcategories are computed in the underlying categories, this implies that the commutative square 
\begin{equation}\label{equ:overcategory-pullback}
\begin{tikzcd}[row sep=0.3cm]
(\cD_{01})_{/(d_0,d_1)}\dar\rar&(\cD_0)_{/(d_0,d_1)}\dar\\
(\cD_1)_{/(d_0,d_1)}\rar&(\cD_0\cup_{\cD_{01}}\cD_1)_{/(d_0,d_1)}
\end{tikzcd}
\end{equation}
of fully faithful inclusions is again a pushout. Therefore the limit $\lim_{(\iota(d)\ra (d_0,d_1))}X(\iota(d))$ is the pullback of the limits of the restriction of the diagram to the full subcategories in \eqref{equ:overcategory-pullback}. The latter all have terminal objects induced by the units of the $(\kappa_i\dashv\partial_i)$-adjunction: $(\epsilon_0,\id)\colon (\kappa_0\partial_0(d_0),d_1)\ra (d_0,d_1)$ is terminal in $(\cD_1)_{/(d_0,d_1)}$, $(\id,\epsilon_1)\colon (d_0,d_1)\ra (d_0,\kappa_1\partial_1(d_1)$ is terminal in $(\cD_0)_{/(d_0,d_1)}$, and $(\epsilon_0,\epsilon_1)\colon (d_0,d_1)\ra (\kappa_0\partial_0(d_0),\kappa_1\partial_1(d_1))$ is terminal in $(\cD_0)_{/(d_0,d_1)}$. Using that limits over categories with terminal objects are given by the value at that terminal object, the claim \ref{enum:right-Kan} follows.

We prove \ref{enum:BC} more generally in stated: assume we have \emph{maps of adjunctions} (maps of bicartesian fibrations over $\Delta^1$ \cite[4.7.4]{LurieHA}, e.g.\, obtained by restricting an adjunction to full subcategories as in \ref{enum:right-Kan})
\begin{equation}\label{equ:maps-of-adjunctions}
\begin{tikzcd}[row sep=0.5cm]\cD'_0  \arrow[swap,d,"\phi_0"] \rar[shift left=-.5ex,swap]{\partial'_0} & \cD'_{01}  \lar[shift left=-.5ex,swap]{\kappa'_0} \arrow[d,"\phi_{01}"]   \\
\cD_0 \rar[shift left=-.5ex,swap]{\partial_0} & \cD_{01} \lar[shift left=-.5ex,swap]{\kappa_0} \end{tikzcd}
\quad \text{and}\quad
\begin{tikzcd}[row sep=0.5cm]  \cD'_{01} \rar[shift left=.5ex]{\kappa'_1} \arrow[swap,d,"\phi_{01}"] \phi_0 & \cD'_1 \dar{\phi_1}  \lar[shift left=.5ex]{\partial'_1}  \\s
 \cD_{01} \rar[shift left=.5ex]{\kappa_1} & \cD_1  \lar[shift left=.5ex]{\partial_1}  \end{tikzcd} \end{equation}
where the $\kappa_i$ and $\kappa'_i$ are fully faithful. In addition to commutativity of the four squares obtained from \eqref{equ:maps-of-adjunctions} by forgetting the $\kappa_i$ and $\kappa_i'$ or the $\partial_i$ and $\partial_i'$, maps of adjunctions give compatibilities between the unit transformations, which yields a map of commutative diagrams 
\[\text{from}\quad \begin{tikzcd}[row sep=0.5cm]\cD'_0 \dar{\id} & \cD'_{01} \rar[shift left=.5ex]{\kappa'_1} \lar[shift left=-.5ex,swap]{\kappa'_0} \dar{\id} & \cD'_1 \dar{\id} \\
\cD'_0 \rar[shift left=-.5ex,swap]{\partial'_0} & \cD'_{01} & \cD'_1  \lar[shift left=.5ex]{\partial'_1} \end{tikzcd} \quad \text{to} \quad \begin{tikzcd}[row sep=0.5cm] \cD_0 \dar{\id} & \cD_{01} \rar[shift left=.5ex]{\kappa_1} \lar[shift left=-.5ex,swap]{\kappa_0} \dar{\id} & \cD_1 \dar{\id} \\
\cD_0 \rar[shift left=-.5ex,swap]{\partial_0} & \cD'_{01} & \cD_1  \lar[shift left=.5ex]{\partial_1}.\end{tikzcd}\]
This yields an equivalence $\iota \phi_\sqcup \simeq \phi_\times \iota'$ where $\phi_\sqcup$ and $\phi_\times$ are the induced functors between the pushouts of the top rows or the pullbacks of the bottom rows, respectively. The more general version of \ref{enum:BC} we show is that the diagram
\[
\begin{tikzcd}[row sep=0.5cm]
	\PSh(\cD_0\cup_{\cD_{01}}\cD_1)\dar[swap]{\iota_*}\rar{\phi_\sqcup^*}&\PSh(\cD'_0\cup_{\cD'_{01}}\cD'_1)\dar{\iota'_*}\\
	\PSh(\cD_0\times_{\cD_{01}}\cD_1)\rar{\phi_\times^*}&\PSh(\cD'_0\times_{\cD'_{01}}\cD'_1),
\end{tikzcd}
\]
commutes, that is, the Beck--Chevalley transformation $\phi_\times^*\iota_*\ra \iota'_*\phi_\sqcup^*$ is an equivalence. It suffices to prove this after precomposition with $\iota^*$ since the latter is essentially surjective as a result of \ref{enum:pushout-ff}. Then the transformation has the form $\phi_\times^*\iota_*\iota^*\ra \iota'_*\phi_\sqcup^*\iota^*\simeq \iota'_*(\iota')^*\phi_\times $ and the claim that it is an equivalence follows from the description of $\iota'_*(\iota')^*$ and $\iota_*\iota^*$ from \ref{enum:right-Kan}.

Finally, \ref{enum:restricted-yoneda} follows by combining \ref{enum:right-Kan}, observation \ref{enum:obs-i}, and the fact that mapping spaces in pullbacks of categories are the pullbacks of the mapping spaces.
\end{proof}

\subsection{Proof of \cref{thm:gluing}}
Equipped with \cref{lem:pushout-lem}, we now move towards the proof of \cref{thm:gluing}. Setting $\smash{\DiscInf^\midman_d\coloneqq \DiscInf^\partial_d\times_{\DiscInf_{d-1}}\DiscInf^\partial_d}$, the gluing functor $\smash{\gamma\colon \ManInf_d^\midman\ra \ManInf_{d}}$ from \cref{sec:glue-disc-ps} restricts to $\smash{\gamma\colon \DiscInf^\midman_d\ra \DiscInf_d}$, so writing $\smash{\DiscInf^\midman_{d,\le\bullet}\coloneqq\gamma^{-1}(\DiscInf_{d,\le\bullet})}$, we get a map $\smash{\gamma\colon \DiscInf^\midman_{d,\le\bullet}\ra \DiscInf_{d,\le\bullet}}$ of towers.  There is also a commutative square of towers
\[
\begin{tikzcd}[column sep=1.5cm,row sep=0.5cm]
\DiscInf_{d-1,\le \bullet}\rar{\kappa}\arrow[d,"\kappa",swap]&\DiscInf_{d,\le\bullet}^\partial\dar{(\id,\kappa\partial(-))}\\
\DiscInf_{d,\le\bullet}^\partial\arrow[r,"{(\kappa\partial(-),\id)}",swap]&\DiscInf^\midman_{d,\le\bullet}
\end{tikzcd}
\]
which induces a map of towers $j\colon \DiscInf^{\ell r}_{d,\le\bullet}\ra \DiscInf^\midman_{d,\le\bullet}$ out of the tower of levelwise pushouts of categories $\smash{\DiscInf^{\ell r}_{d,\le\bullet}\coloneqq \DiscInf_{d,\le\bullet}^\partial\cup_{\DiscInf_{d-1,\le \bullet}}\DiscInf_{d,\le\bullet}^\partial}$. As further preparation for the proof of \cref{thm:gluing}, we will show that left Kan extension along $\gamma$ and right Kan extension along $j$ interact well with restriction along the inclusions between the $\DiscInf$-subcategories:

\begin{lem}\label{lem:gamma-j-towers}The following diagrams commute for $1\le k\le l\le\infty$
\begin{equation}
\begin{tikzcd}\PSh(\DiscInf_{d,\le l}^{\midman})\arrow[d,swap,"\gamma_!"]\rar{\inc^*}&\PSh(\DiscInf_{d,\le k}^{\midman})\dar{\gamma_!}\\
\PSh(\DiscInf_{d,\le l})\rar{\inc^*}&\PSh(\DiscInf_{d,\le k})\end{tikzcd}
\quad
\begin{tikzcd}
\PSh(\DiscInf^{\ell r}_{d,\le l})\arrow[d,swap,"j_*"]\rar{\inc^*}&\PSh(\DiscInf^{\ell r}_{d,\le k})\dar{j_*}\\
\PSh(\DiscInf_{d,\le l}^{\midman})\rar{\inc^*}&\PSh(\DiscInf_{d,\le k}^{\midman})
\end{tikzcd}  
\end{equation}
i.e.\,the Beck--Chevalley transformations $\gamma_!\inc^*\ra\inc^*\gamma_!$ and $\inc^*j_*\ra j_*\inc^*$ are equivalences.
\end{lem}
\begin{proof}
The claim regarding the first square follows from the same argument as for \cite[Lemma 4.4]{KKoperadic} by replacing the role of $\Env(\cO)$ and $\Env(\cP)$ in loc.cit.\,by the categories $\DiscInf_{d}^{\midman}$ and $\DiscInf_{d}$ respectively and using that any embedding between manifolds factors uniquely as the composition of an embedding which is surjective on path-components with an inclusion of path-components. Regarding the second square, we first recall that $\smash{\DiscInf^{\midman}_{d,\le n}}$ was defined as a  full subcategory of $\smash{\DiscInf^{\midman}_{d}=\DiscInf^\partial_d\times_{\DiscInf_{d-1}}\DiscInf^\partial_d}$, namely the preimage of $\DiscInf_{d,\le n}$ under $\gamma$. This is contained in the full subcategory $\smash{\DiscInf^\partial_{d,\le n}\times_{\DiscInf_{d-1,\le n}}\DiscInf^\partial_{d,\le n}\subset \DiscInf^{\midman}_{d}}$, but is smaller. However, as right Kan extension along a fully faithful functor is fully faithful, it suffices to prove commutativity of the second square when replacing $\DiscInf^{\midman}_{d,\le n}$ with $\smash{\DiscInf^\partial_{d,\le n}\times_{\DiscInf_{d-1,\le n}}\DiscInf^\partial_{d,\le n}}$ for $n=k$ and $l$ respectively, and then the claim becomes an instance of \cref{lem:pushout-lem} \ref{enum:BC}.
\end{proof}
As a result of \cref{lem:gamma-j-towers}, we have maps of towers \begin{equation}\label{equ:j-gamma-towers} j_*\colon \PSh(\DiscInf_{d,\le \bullet}^{\ell r})\lra \PSh(\DiscInf^\midman_{d,\le \bullet})\quad\text{and}\quad\gamma_!\colon \PSh(\DiscInf_{d,\le \bullet}^{\midman})\ra \PSh(\DiscInf_{d,\le \bullet}).\end{equation} Now consider the following (potentially non-commutative) diagram of towers of categories
\begin{equation}\label{equ:final-tower-diagram}\begin{tikzcd}[row sep=0.4cm]
\ManInf_d^\midman\rar{y}\arrow[d,"\gamma",swap]& \PSh(\ManInf_d^\midman)\dar{\gamma_!}\rar{(\iota^\midman)^*}&\PSh(\DiscInf^\midman_{d,\le\bullet})\dar{\gamma_!}\rar{j^*}&\PSh(\DiscInf^{\ell r}_{d,\le\bullet})\arrow[dl,"\gamma_!j_*", bend left=10]\\
\ManInf_d\rar{y}&\PSh(\ManInf_d)\rar{\iota^*}&\PSh(\DiscInf_{d,\le\bullet})&&.
\end{tikzcd}
\end{equation}
The bottom row agrees by definition with the bottom map in the claim of \cref{thm:gluing}. Moreover, recalling the definition of the top map $E^{\midman}$ in \cref{thm:gluing} and using that $\PSh(-)\colon \Cat^\op\ra\Cat$ sends pushouts to pullbacks, one sees that the top row in \eqref{equ:final-tower-diagram} agrees with $E^{\midman}$, so setting $\gamma^{\DiscInf}\coloneqq \gamma_!j_*$ in order to prove \cref{thm:gluing} it suffices to show that the outer two compositions $\ManInf_d^\midman\ra \PSh(\DiscInf_{d,\le\bullet})$ are equivalent as maps of towers. The two ingredients for this are:
\begin{lem}\label{lem:unit-on-yoneda}The unit $\id\ra j_*j^*$ is an equivalence on the essential image of $(\iota^\midman)^* \circ y$.
\end{lem}
\begin{proof}
By an application of the triangle identities for adjunctions, one sees that it suffices to show the claim after replacing the tower $\DiscInf^{\midman}_{d,\le \bullet}$ by the tower $\smash{\DiscInf^\partial_{d,\le \bullet}\times_{\DiscInf_{d-1,\le \bullet}}\DiscInf^\partial_{d,\le \bullet}}$ in which case it is an instance of \cref{lem:pushout-lem} \ref{enum:restricted-yoneda}.
\end{proof}
\begin{lem}\label{lem:descent-gluing}The second square in \eqref{equ:final-tower-diagram} commutes, i.e.\,the Beck--Chevalley transformation $\gamma_!(\iota^\midman)^*\ra \iota^*\gamma_!$ is an equivalence.
\end{lem}
\begin{proof}
Since we have already seen in \eqref{equ:j-gamma-towers} that $\gamma_!$ is a map of towers, it suffices to prove the statement for $\bullet=\infty$. Both sides of the Beck--Chevalley transformation $\delta\colon \gamma_!(\iota^{\midman})^*\ra \iota^*\gamma_!$ commute with colimits, so since presheaf-categories are generated under colimits by representables, it suffices in view of \cref{rem:objects-in-manmid} to show that for each $d$-manifold $M$ with no boundary and codimension $0$ submanifolds $M_\ell,M_r\subset M$ that intersect in their common boundary $P\coloneq \partial M_{\ell}=\partial M_r$, the transformation $\delta$ is an equivalence on ${\Map_{\ManInf_d^\midman}(-,M^\midman)\in \PSh(\ManInf_d^\midman)}$ for $M^\midman\coloneq (M_\ell,M_r,\id)\in \ManInf_d^\midman$. To do so, we consider the poset $\cO^{\midman}_{\infty}(M)$ of open subsets $O\subset M$ such that the manifolds with boundary $O\cap M_\ell$ and $O\cap M_r$ are both contained in $\DiscInf^\partial_d$. Inclusion gives a functor $\cO^{\midman}_{\infty}(M)\ra (\ManInf^{\midman}_d)_{/M^\midman}$ that sends $O\in \cO^{\midman}_{\infty}(M)$ to $O^{\midman}\coloneq (O\cap M_\ell,O\cap M_r,\id_{O\cap P})\subset M^\midman$. Now consider the commutative square in $\PSh(\DiscInf_d)$
\[
\begin{tikzcd}[row sep=0.5cm]
\gamma_!(\iota^{\midman})^*\big(\colim_{O\in\cO_\infty^\midman(M)}\Map_{\ManInf_d^\midman}(-,O^\midman)\big)\arrow[r,"\circled{1}","\delta"']\dar{\circled{2}}&\iota^*\gamma_!\big(\colim_{O\in\cO_\infty^\midman(M)}\Map_{\ManInf_d^\midman}(-,O^\midman)\big)\dar{\circled{3}}\\
\gamma_!(\iota^{\midman})^*\big(\Map_{\ManInf_d^\midman}(-,M^\midman))\arrow[r,"\circled{4}","\delta"']&\iota^*\gamma_!\big(\Map_{\ManInf_d^\midman}(-,M^\midman)).
\end{tikzcd}
\]
We will show that $\circled{1}$-$\circled{3}$ are equivalences, which implies that $\circled{4}$ is one as well, so the claim will follow. To do this, we repeatedly use the facts that restriction and left Kan extension preserve colimits and that left Kan extension preserves representables. From these facts, together with the observation that $\smash{O^\midman\in\DiscInf_d^\midman}$, one sees that $\circled{1}$ is equivalent to the identity on $\smash{\colim_{O\in\cO_\infty^\midman(M)}\Map_{\DiscInf_d}(-,O)}$, so in particular an equivalence. Using the above facts again, the map $\circled{3}$ is equivalent to the map  $\smash{\colim_{O\in\cO^\midman_\infty(M)}\Map_{\ManInf_d}(\iota(-),O)\ra \Map_{\ManInf_d}(\iota(-),M)}$. Since $\smash{\{O\}_{O\in\cO_{\infty}^{\midman}(M)}}$ is a complete Weiss $\infty$-cover of $M$ in the sense of \cite[Definition 6.3]{KnudsenKupers}, the proof of Lemma 6.4 loc.cit.\ shows that this map is indeed an equivalence. Since it will be relevant later, recall that the key step in the proof of this result in loc.cit.\ is an application of \cite[Proposition 4.6 (c)]{DuggerIsaksen} to the open cover $\{F_n(O)\}_{O\in\cO_{\infty}^{\midman}(M)}$ of the space $F_n(M)$ of ordered configurations of $n$ points in $M$, for $n\ge0$, using that this is a complete cover in the sense of Definition 4.5 loc.cit.. We will now show by a similar argument that the map $\smash{\colim_{O\in\cO^\midman_\infty(M)}\Map_{\ManInf^\midman_d}(\iota^\midman(-),O^\midman)\ra \Map_{\ManInf^\midman_d}(\iota^\midman(-),M^\midman)}$ is an equivalence, which will imply that $\gamma_!(-)$ of it---which is precisely \circled{2}---is an equivalence as well, so the claim will follow. Adapting the argument in the proof of \cite[Lemma 6.4]{KnudsenKupers} to this map reduces the claim to showing that for $r,s,t\ge0$ the open cover $\smash{\{F_{r,s,t}(O)\}_{O\in\cO_{\infty}^{\midman}(M)}}$ of the space $F_{r,s,t}(M)$ of ordered configurations of $r+s+t$ points in $M$ where the first $r$ points lie in $\interior(M_\ell)$, the second $s$ points in $P$, and the final $t$ points in $\interior(M_r)$, is a complete cover. But we have $F_{r,s,t}(O)=F_{r+s+t}(O)\cap F_{r,s,t}(M)$, so the claim follows from the corresponding fact for $\{F_{r+s+t}(O)\}_{O\in\cO_{\infty}^{\midman}(M)}$ we have already used above by observing that complete covers are preserved by taking intersections with a fixed subspace.
\end{proof}

We end the section by finishing the proof of \cref{thm:gluing} and deducing \cref{thm:gluing-isotopy}.
\begin{proof}[Proof of \cref{thm:gluing}]
By the discussion above the diagram \eqref{equ:final-tower-diagram}, it suffices to show that the outermost two compositions in it agree. As a result of \cref{lem:unit-on-yoneda}, it suffices to show that the two leftmost squares in the diagram commute. For the second square, this is \cref{lem:descent-gluing} and for the first square it is an instance of the naturality of the Yoneda embedding.
\end{proof}

\begin{proof}[Proof of \cref{thm:gluing-isotopy}]For manifolds $M_i$ and diffeomorphism $\varphi_i$ as in the statement, we have that $E_{M_0\cup_{\varphi_i}M_1}$ in $\PSh(\DiscInf_{d,\le k})$ for $i=0,1$ are the values of $(M_0,M_1,\varphi_i)\in\ManInf_d^{\midman}$ under the counterclockwise composition in the square of \cref{thm:gluing} for $\bullet=k$, so by commutativity they are also the values under the clockwise composition. But the assumption implies that their images $(E^\partial_{M_0},E^\partial_{M_0},E_{\varphi_i})$ under the top horizontal map $E^\midman$ agree, so the claim follows. The addendum follows in the same way, using a variant of \cref{thm:gluing} for oriented manifolds which is proved in the same way as the non-oriented version (This is also a special case of the version for general tangential structures as alluded to in \cref{rem:extensions-gluing} \ref{enum:extensions-gluing-ii}).
\end{proof}

\section{Stable collapse maps of $\DiscInf$-presheaves}\label{sec:normal-invariants}
In this section we establish another one of the three ingredients that we assumed in \cref{sec:main-ideas}, namely \cref{thm:collapse-zigzag} regarding the naturality of  Atiyah duality in equivalences of $2$-truncated presheaves. We adopt the notation from \cref{sec:coker-j-spheres}.

\begin{nrem}[Idea of proof]Informally speaking, \cref{thm:collapse-zigzag} says that the stable collapse map $\bfS\ra \Th(-TM)$ of a closed $d$-manifold $M$ can be extracted from the associated $2$-truncated $\DiscInf$-presheaf $E_M$, so the idea of proof is to find a suitable construction of the stable collapse map for manifolds which can be imitated for  $2$-truncated $\DiscInf$-presheaves. One way to construct the stable collapse map $\bfS\ra \Th(-TM)$ for a closed $d$-manifold $M$ is to choose an embedding $M\subset\bfR^N$ for $N\gg0$ and use the Pontryagin--Thom construction, but this way appears to not easily be replicated for $\DiscInf$-presheaves. There is, however, another way to construct the stable collapse map which does not require the choice of an embedding in $\bfR^N$, namely by using a Poincaré embedding structure on the diagonal $\Delta\colon M\hookrightarrow M\times M$ (see \cref{ex:con-manifold} below). The proof of \cref{thm:collapse-zigzag} is basically to show that this other way \emph{can} be mimicked for $\DiscInf$-presheaves: we first explain that this construction of the stable collapse map is a special case of a general construction using parametrised spectra whose input is a commutative square in spaces with left vertical map a spherical fibration $\pi\colon S(\xi)\ra B$ associated to a vector bundle $\xi$ and with bottom horizontal map is the diagonal of $B$, and whose output gives a zig-zag of spectra between $\bfS$ and $\Th(-\xi)$ (see Constructions \ref{con:pre-collapse} and \ref{con:square-spherical-fibration}). We will then see, firstly, that this recovers the stable collapse map $\bfS\ra\Th(-TM)$ if one inputs the square underlying the Poincaré embedding structure of the diagonal of a closed $d$-manifold $M$ (see \cref{ex:con-manifold}), secondly, that there is a square associated to any presheaf $X\in\PSh(\DiscInf_{d,\le 2})$ (see  \cref{ex:square-presheaf}) to which one can apply the construction, and, thirdly, that these squares agree in the case $X=E_M$ (see \cref{lem:zig-zag-agrees}). This will then show that the zig-zag of spectra obtained in this manner from the square associated to $X\in\PSh(\DiscInf_{d,\le 2})$ has the properties claimed in \cref{thm:collapse-zigzag}.
\end{nrem}

\subsection{Proof of \cref{thm:collapse-zigzag}} We begin by fixing some notation:
\begin{enumerate}[leftmargin=0.6cm, label=(\alph*)]
	\item For a space $B$, we write $\cS_{/B}$,  $(\cS_{/B})_*$, and $\Sp_B$ for the $\infty$-categories of spaces over $B$, retractive spaces over $B$ (spaces over $B$ equipped with a section), and parametrised spectra over $B$, respectively. By straightening, they are equivalent to the functor categories $\Fun(B,\cS)$, $\Fun(B,\cS_*)$, and $\Fun(B,\Sp)$, respectively. 
	\item We denote various fibrewise constructions in $\cS_{/B}$,  $(\cS_{/B})_*$, and $\Sp_B$ by adding a $B$-subscript. For instance, $\Sigma^\infty_B (-)\colon (\cS_{/B})_*\ra \Sp_B$ denotes taking fibrewise suspension spectrum, $(-)_{+,B}\colon \cS_{/B}\ra (\cS_{/B})_*$ fibrewise adding a disjoint basepoint,  $(-)\otimes_B(-)\colon \Sp_B\times \Sp_B\ra \Sp_B$ fibrewise tensor product,  $\IntMap_B(-,-)\colon \Sp_B^\op\times \Sp_B\ra \Sp_B$ fibrewise mapping spectra, and $D_B(-)=\IntMap_B(-,\bfS_B)\colon \Sp_B^\op\ra \Sp_B$ fibrewise Spanier--Whitehead dual given by taking fibrewise mapping spectra into the constant parametrized spectrum $\bfS_B$ with fibre $\bfS$.
	\item For a map of spaces $f\colon B\ra C$, the restriction functor $f^*\colon \Sp_C\ra \Sp_B$ has a left adjoint $f_!\colon  \Sp_B\ra \Sp_C$, given by left Kan extension under straightening. In particular, for the constant map $f=p\colon B\ra \ast$ the adjoint $p_!\colon \Sp_B\ra \Sp_*=\Sp$ is given by taking colimits. 
	\item\label{enum:fibrewise-cone} For a map $\varphi\colon V\ra W$ in $\cS_{/B}$, we denote the pushout in $\cS$ of $V\ra W$ along $V\ra B$ as $\smash{C_B(\varphi)}$ and think of it as a relative mapping cone. Note that $C_B(\varphi)$ comes with canonical maps to and from $B$, which turn it into a retractive space. This construction is natural in that a commutative square in $\cS_{/B}$
	\begin{equation}\label{equ:square-over-B}
		\begin{tikzcd}[row sep=0.5cm]
			V'\arrow[d,"\varphi",swap]\rar&V\dar{\varphi'}\\
			W'\rar&W
		\end{tikzcd}
	\end{equation}
	induces a map $C_B(\varphi')\ra C_B(\varphi)$ in $(\cS_{/B})_*$. For $V'=\varnothing$ and $W'=W$ this in particular gives a map $W_{+,B}\ra C_B(\varphi)$ in $(\cS_{/B})_*$. An important class of examples for us is when $\varphi$ is the projection $S(\xi)\ra B$ of the $(d-1)$-dimensional spherical fibration of a $d$-dimensional vector bundle $\xi$ over $B$ given by removing the $0$-section. In this case, $C_B(\varphi)\ra B$ agrees with the retractive space given as the $d$-dimensional spherical fibration obtained from $\xi$ by fibrewise one-point compactification, or equivalently obtained from $S(\xi)$ by fibrewise cone. Then $\smash{\Sigma^\infty_BC_B(\varphi)}$ is the usual parametrised spectrum with fibres $\bfS^d$ associated to a $d$-dimensional vector bundle $\xi$, so $\smash{p_!(\Sigma^\infty_BC_B(\varphi))\simeq \Th(\xi)}$ is its Thom-spectrum.
\end{enumerate}

\begin{construction}\label{con:pre-collapse}Given a space $B$, we view $B\times B$ as a space over $B$ via the projection to the first coordinate. Under the straightening equivalence, this corresponds to the constant functor $\mathrm{const}_B\in\Fun(B,\cS)$, so the result $\Sigma^\infty_B((B\times B)_{+,B})\in \Sp_B$ of adding a fibrewise disjoint basepoint and then taking fibrewise suspension spectra corresponds to the constant functor $\mathrm{const}_{\Sigma_+^\infty B}\in\Fun(B,\Sp)$. The latter is also the straightening of $p^*(\Sigma_+^\infty B)\in \Sp_B$ where $p\colon B\ra \ast$ is the constant map, so we have a canonical equivalence $p^*(\Sigma^\infty_+B)\simeq \Sigma^\infty_B ((B\times B)_{+,B})$. Now given $V\in\cS_{/B}$ and a map $\varphi\colon V\ra B\times B$ in $\cS_{/B}$, we may apply $\Sigma^\infty_B(-)$ to the map $(B\times B)_{+,B}\ra C_B(\varphi)$ from \ref{enum:fibrewise-cone} to arrive at a map $p^*(\Sigma^\infty_+B)\simeq \Sigma^\infty_B ((B\times B)_{+,B})\ra \Sigma^\infty_BC_B(\varphi)$. Combining this with the fibrewise evaluation map $D_B(\Sigma^\infty_BC_B(\varphi))\otimes_B \Sigma^\infty_BC_B(\varphi) \ra p^*(\bfS)$, we obtain a map $D_B(\Sigma^\infty_BC_B(\varphi))\otimes_B p^*(\Sigma^\infty_+B)\ra p^*(\bfS)$ which yields using the tensor-hom adjunction  a map $D_B(\Sigma^\infty_BC_B(\varphi))\ra \IntMap_B(p^*(\Sigma^\infty_+B), p^*(\bfS))\simeq p^*\IntMap(\Sigma^\infty_+B, \bfS)=p^*D(\Sigma^\infty_+B)$ which in turn yields via the $(p_!\dashv p^*)$-adjunction a map of spectra of the form
	\begin{equation}\label{equ:canmap}p_!D_B(\Sigma^\infty_BC_B(\varphi))\lra D(\Sigma^\infty_+B).\end{equation}
	
\end{construction}
\begin{construction}\label{con:square-spherical-fibration}Suppose we are given a $d$-dimensional vector bundle $\xi\colon B\ra \BO(d)$ over a space $B$ and a commutative square in $\cS$ of the form
	\begin{equation}\label{equ:square-general}
		\begin{tikzcd}[row sep=0.5cm]			S(\xi)\dar[swap]{\pi} \rar& C\dar{i}\\
			B\rar{\Delta}&B\times B
		\end{tikzcd}
	\end{equation}
	where $S(\xi)$ is the $(d-1)$-dimensional spherical fibration induced by $\xi$, the bottom horizontal map is the diagonal, and $i$ is any map. Viewing this square as a square in $\cS_{/B}$ via the projection of $B\times B$ onto the first coordinate, 
the naturality in \ref{enum:fibrewise-cone} yields a map $C_B(S(\xi))\ra C_B(i)$ and thus by applying $p_!D_B(\Sigma_B^\infty(-))$ a map $p_!D_B(\Sigma_B^\infty C_B(i))\ra p_!D_B(\Sigma_B^\infty C_B(S(\xi)))\simeq \Th(-\xi)$ (see \ref{enum:fibrewise-cone} for the final equivalence). This features in a natural zig-zag of spectra
	\begin{equation}\label{equ:collapse-zig-zag}
		\bfS\simeq D(\Sigma_+^\infty \ast)\lra D(\Sigma_+^\infty B) \lla p_!D_B(\Sigma^\infty_B C_B(i))\lra \Th(-\xi)\end{equation}
	where the first map results from applying $D(\Sigma_+^\infty(-))$ to the map $p\colon B\ra \ast$, the second map is an instance of \eqref{equ:canmap}, and the final map is the one we just described.
\end{construction}
	
\begin{ex}\label{ex:con-manifold} For us, the most important example of a square as in \eqref{equ:square-general} is the square in $\cS$
	\begin{equation}\label{equ:square-manifold}\begin{tikzcd}[row sep=0.5cm] S(TM) \rar \dar[swap]{\pi} & M \times M \backslash \Delta \dar{\subset} \\
		M \rar{\Delta} & M \times M \end{tikzcd}\end{equation}
		where $B=M$ is a $d$-manifold, $\xi=TM$ is its tangent bundle and $i\colon C\ra B\times B$ is the inclusion $M\times M\backslash \Delta \subset M\times M$ of the complement of the diagonal. This square arises as follows: writing $0_M\subset TM$ for the $0$-section, a choice of tubular neighbourhood of the diagonal $M\subset M\times M$ yields an embedding of manifold pairs $(TM,TM\backslash0_M)\hookrightarrow (M\times M,M\times M\backslash\Delta)$ from which one obtains \eqref{equ:square-manifold} by combining the equivalence of pairs of spaces $(D(M),S(TM))\simeq (TM,TM\backslash0_M)$ induced by the inclusion together with the equivalence $D(TM)\simeq M$ given by the projection.
		
		 If $M$ is a \emph{closed} manifold, then the wrong-way map in the instance of the zig-zag \eqref{equ:collapse-zig-zag} associated to \eqref{equ:square-manifold} turns out to be an equivalence (see e.g.\,\cite[Proposition 4.12]{AndoBlumbergGepner}) and the map $\bfS\ra \Th(-TM)$ resulting from choosing an inverse is equivalent to the stable collapse map (see e.g.\,Section 4.3 loc.cit.\ or \cite[Section 4.2]{NaefSafronov} for explanations of this).
\end{ex}

\begin{ex}\label{ex:square-presheaf}Any $k$-truncated presheaf $X\in \PSh(\DiscInf_{d,\le k})$ for $k\ge2$ also yields an example of a square as in \eqref{equ:square-general}: using the notation $\vert X\rvert$ and $\xi_X$ from \cref{sec:coker-j-spheres} and writing $\Emb(-,-)$  and $\Diff(-)$ for the mapping and automorphism spaces in $\DiscInf_{d,\leq k}$, we have a square \begin{equation}
	\begin{tikzcd}[row sep=0.5cm]\label{equ:square}
		S(\xi_X)\simeq \Big(\frac{\Emb(\bfR^d\sqcup \bfR^d,\bfR^d)}{\Diff(\bfR^d)^{\times 2}}\times_{\Diff(\bfR^d)}X(\bfR^d)\Big) \rar\arrow[d,swap,"\pi"]& \frac{X(\bfR^d\sqcup \bfR^d)}{\Diff(\bfR^d)^{\times 2}}\dar\\
		\lvert X\rvert\rar{\Delta}&\lvert X\rvert\times \lvert X\rvert.
	\end{tikzcd}
\end{equation}
 obtained as follows: the map $\Emb(\bfR^d\sqcup \bfR^d,\bfR^d)\ra \Emb(\bfR^d,\bfR^d)^{\times 2}$ given by precomposition with the inclusions, and the functoriality of $X$, yields a map of pairs
\begin{equation}\label{equ:map-pairs-square}
\begin{tikzcd}[row sep=0.2cm]\Big(\big(\frac{\Emb(\bfR^d,\bfR^d)}{\Diff(\bfR^d)}\big)^{\times 2}\times_{\Diff(\bfR^d)}X(\bfR^d),\frac{\Emb(\bfR^d\sqcup \bfR^d,\bfR^d)}{\Diff(\bfR^d)^{\times 2}}\times_{\Diff(\bfR^d)}X(\bfR^d)\Big)\dar\\
\Big(\big(\frac{X(\bfR^d)}{\Diff(\bfR^d)}\big)^{\times 2},\frac{X(\bfR^d\sqcup \bfR^d)}{\Diff(\bfR^d)^{\times 2}}\Big)
\end{tikzcd}
\end{equation}
from which one obtains \eqref{equ:square} by using $\Emb(\bfR^d,\bfR^d)/\Diff(\bfR^d)\simeq \ast$, $\lvert X\rvert=*\times_{\Diff(\bfR^d)}X(\bfR^d)$ together with the $\GL_d(\bfR)$-equivariant equivalence $\smash{\Emb(\bfR^d\sqcup \bfR^d,\bfR^d)/\Diff(\bfR^d)^{\times 2} \simeq \bfR^d\backslash\{0\}}$ to identify the source in \eqref{equ:map-pairs-square} with the pair $(S(\xi_X),\vert X\rvert)$ given by the projection. 
\end{ex}

The two examples Examples \ref{ex:con-manifold} and \ref{ex:square-presheaf} turn out to be closely related: 
\begin{lem}\label{lem:zig-zag-agrees}
For a smooth $d$-manifold $M$, there is an equivalence of squares between \eqref{equ:square-manifold} and \eqref{equ:square} for $X=E_M$. More specifically, there is an equivalence of spherical fibrations $S(TM)\simeq S(\xi_{E_M})$ such that the resulting equivalences between the left vertical and bottom horizontal maps in \eqref{equ:square-manifold} and \eqref{equ:square} can be lifted to an equivalence of squares.
\end{lem}
\begin{proof}
To see this, one first uses that evaluation at the centre induces an equivalence between the map of pairs \eqref{equ:map-pairs-square} for $X=E_M=\Emb(-,M)$ and the map of pairs
\begin{equation}\label{equ:map-pairs-square-2}\hspace{-0.3cm}
\begin{tikzcd}[row sep=0.2cm]
\Big(\Emb(\ast,\bfR^d)^{\times 2}\times_{\Diff(\bfR^d)}\Emb(\bfR^d,M),\Emb(\ast\sqcup \ast,\bfR^d)\times_{\Diff(\bfR^d)}\Emb(\bfR^d,M)\Big)\dar\\
 (\Emb(\ast,M)^{\times 2},\Emb(\ast\sqcup \ast,M))=(M\times M,M\times M\backslash \Delta).
 \end{tikzcd}
\end{equation}
induced by composition of embeddings. It thus suffices to construct an equivalence of pairs 
\[(TM,TM\backslash 0_M\})\xra{\simeq} \Big(\Emb(\ast,\bfR^d)^{\times 2}\times_{\Diff(\bfR^d)}\Emb(\bfR^d,M),\Emb(\ast\sqcup \ast,\bfR^d)\times_{\Diff(\bfR^d)}\Emb(\bfR^d,M)\Big)\]
such that its postcomposition with \eqref{equ:map-pairs-square-2} agrees with the map $(TM,TM\backslash 0_M)\ra (M\times M,M\times M\backslash \Delta)$ used in \cref{ex:con-manifold} (involving the choice of a tubular neighbourhood). By the uniqueness of tubular neighbourhoods, we may assume that the tubular neighbourhood $TM\hookrightarrow M\times M$ involved is on the first coordinate given by the projection and on the second coordinate given by a smooth retraction $e\colon TM\ra M$ of the $0$-section that is an embedding when restricted to each individual tangent space (to obtain $e$, fix a Riemannian metric and define $e$ to first fibrewise scale vectors to be shorter than the injectivity radius and then apply the exponential map; see e.g.\,\cite[Section 12]{BrockerJanich}). Such a map $e$ yields an $\GL_d(\bfR)$-equivariant equivalence $\smash{\Fr(TM)\xra{\simeq} \Emb(\bfR^d,M)}$ out of the frame bundle, which together with the $\oO(d)$-equivariant equivalence of pairs $\smash{(\bfR^d,\bfR^{d}\backslash \{0\})\xra{\simeq} (\Emb(*,\bfR^d)^{\times 2},\Emb(*\sqcup *,\bfR^d))}$ given by $\bfR^d\ni x\mapsto(0,x)\in \Emb(*,\bfR^d)^{\times 2}$ yields an equivalence of pairs of the form
\[\begin{tikzcd}[row sep=0.3cm]
\big(\bfR^d\times_{\GL_d(\bfR)}\Fr(TM), \bfR^d\backslash\{0\}\times_{\GL_d(\bfR)}\Fr(TM)\big)\dar{\simeq}\\
\Big(\Emb(*,\bfR^d)^{\times 2}\times_{\Diff(\bfR^d)}\Emb(\bfR^d,M),\Emb(*\sqcup *,\bfR^d)\times_{\Diff(\bfR^d)}\Emb(\bfR^d,M)\Big)\end{tikzcd}\]
Going through the construction, the precomposition of this equivalence with the standard identification of its source with $(TM,TM\backslash0_M)$ has the desired properties.
\end{proof}

\begin{rem}\cref{lem:zig-zag-agrees} is closely related to a previous result by Prigge  \cite[Theorem 7.1.8]{PriggeThesis}. We briefly outline the relation between the two: If $M$ is a closed manifold, then the square \eqref{equ:square-manifold} is a pushout and yields a Poincaré embedding structure on the diagonal of $M$. By \cref{lem:zig-zag-agrees}, the square \eqref{equ:map-pairs-square} yields for $X=E_M$ also a Poincaré embedding structure on the diagonal of $|X|$. Since the latter square is natural $X$, this implies that for any map from a space $Z$ to $B\Aut_{\PSh(\DiscInf_{d,\le 2})}(E_M)$, the associated fibration over $Z$ with fibre $|E_M|\simeq M$ admits a fibrewise Poincaré embedding structure of its fibrewise diagonal. This recovers Theorem 7.1.8 loc.cit..
\end{rem}

\cref{thm:collapse-zigzag} now follows by merely putting things together:

\begin{proof}[Proof of \cref{thm:collapse-zigzag}]Specialising \eqref{equ:collapse-zig-zag} to \cref{ex:square-presheaf} yields a natural zig-zag 
\begin{equation}\label{equ:collapse-zigzag-detail}\bfS\lra D(\Sigma_+^\infty\lvert X\rvert)\lla p_!D_B(\Sigma^\infty C_B(i))\lra \Th(-\xi_X),\end{equation} 
as in the claim. The asserted property of it in the case $X=E_M$ follows by combining \cref{lem:zig-zag-agrees} with the discussion at the end of \cref{ex:con-manifold}.
\end{proof}

\subsection{Further remarks}We end this section with some remarks on the proof of \cref{thm:collapse-zigzag}.


\begin{rem}[Work of Naef--Safronov]\label{rem:naefsafronov}The proof of \cref{thm:collapse-zigzag} is closely related to work of Naef and Safronov \cite{NaefSafronov}. They explain in Section 4.2 loc.cit.\ how the square \eqref{equ:square-manifold} for a closed smooth manifold $M$ can be used to show that the constant parametrised spectrum $\bfS_M$ over $M$ is relatively dualisable in the sense of Definition 1.16 loc.cit.\,with dual $\Sigma^\infty_M S(-TM)$. Our construction above uses the  square to construct the copairing---in their notation $\eta$---while they in Proposition 4.13 loc.cit.~use it to construct the pairing---in their notation $\mathrm{PT}_\Delta$. These contain essentially the same information (see Remark 1.3 loc.cit.).
\end{rem}

\begin{rem}[Topological collapse maps]\label{rem:topological-collapse} \cref{thm:collapse-zigzag} also holds in the topological setting, that is, for $\DiscInf_d$ replaced by the analogous category $\smash{\DiscInf^t_d}$ involving \emph{topological} embeddings, $M$ being a closed \emph{topological} $d$-manifolds $M$, and $\xi_X$ being an $\bfR^d$-bundle (a fibre bundle with fibre $\bfR^d$) instead of a vector bundle, from which one obtains a spherical fibration by removing a section. In fact, the version \cref{thm:collapse-zigzag} for smooth manifolds can be deduced from that for topological ones by left Kan extending along the forgetful functor $\DiscInf^t_{d,\le 2}\ra \DiscInf_{d,\le 2}$.

We briefly explain how to obtain the claimed extension of  \cref{thm:collapse-zigzag} to the topological setting: \cref{con:square-spherical-fibration} goes through verbatim for $\bfR^d$-bundles. In \cref{ex:con-manifold} one replaces $TM$ by the \emph{topological} tangent bundle $T^tM$, obtained by choosing using Kister's theorem \cite{Kister} a normal $\bfR^d$-bundle inside the normal microbundle of the diagonal $M \subset M \times M$; this yields a tubular neighbourhood with the property used in \cref{lem:zig-zag-agrees} by construction. With this choice, the proof of \cref{lem:zig-zag-agrees} goes through with minor changes. Finally, the reference \cite[Proposition 4.12]{AndoBlumbergGepner} we cite in \cref{ex:con-manifold} smooth manifolds, but the proof can be extended to the topological setting:  

One picks a locally flat embedding $M \subset \bfR^{d+k}$ with a normal $\bfR^k$-bundle $\nu_M$ for $k \gg 0$ \cite[Theorem (B)]{HirschNormal}, which is unique up to isotopy after further stabilisations by Theorem (C) loc.cit.. The Pontryagin--Thom construction applied to $M \subseteq \bfR^{d+k}$ with normal bundle $\nu_M$ gives a pointed map $S^{d+k} \to \Th(\nu_M)$, which yields the stable collapse map $\bfS \to \Th(-TM)$. Performing the Pontryagin--Thom construction to both the inclusion $M \times M \subset \bfR^{d+k} \times M$ with normal bundle $\nu_M,$ as well as its precomposition with the diagonal inclusion gives a map $\Th(\nu_M) \wedge M_+ \to S^{d+k} \wedge M_+$. Postcomposing with the projection to $S^{d+k}$, yields a pairing $\Th(-T^tM) \otimes \Sigma^\infty M_+ \to \bfS$, which is the evaluation map that exhibits $\Th(-T^t M)$ as the Spanier--Whitehead dual of $\Sigma^\infty M_+$. Using this, the proof of \cite[Proposition 4.12]{AndoBlumbergGepner} goes through and shows that the wrong-way map in \eqref{equ:collapse-zig-zag} is in the case of \cref{ex:con-manifold} an equivalence and that the constructed map $\bfS \to \Th(-T^tM)$ agrees with the stable collapse map.
\end{rem}

\section{Finite residuals of mapping class groups}\label{sec:fr}
In this section we establish the remaining ingredient that we assumed in \cref{sec:main-ideas} and we have not proved yet, namely \cref{thm:fr-prelim} regarding the finite residual of the group
 \[\Gamma_g\coloneq \pi_0\,\Diff^{\plus}(W_g)\] of isotopy classes of orientation-preserving diffeomorphisms (we omit the dependence on $n$ in the notation) of the $g$-fold connected sum $W_g\coloneq \sharp^g(S^n\times S^n)$ for $n\ge3$ odd. The main input for this is an algebraic description of these mapping class groups which was established in \cite{Krannich-mcg}. We will also use the following properties of finite residuals $\fr(G)$ of discrete groups $G$ (see \cref{sec:e-of-bp} for the definition):

\begin{lem}\label{lem:grp-thy-lem}\ 
\begin{enumerate}[leftmargin=*]
\item \label{lem:grp-thy-lem:functoriality} For a group homomorphism $\varphi\colon E\ra E'$, we have $\fr(E)\subseteq \varphi^{-1}(\fr(E'))$. 
\item\label{lem:grp-thy-lem:iii} For a monomorphism of groups $\varphi\colon E\hookrightarrow E'$ such that its image $\varphi(E)\le E'$ has finite index, we have $\fr(E)=\varphi^{-1}(\fr(E'))$.
\item\label{lem:grp-thy-lem:ex} The following classes of groups are residually finite:
\begin{enumerate}
\item\label{ex:res-fin-fg} finitely generated abelian groups,
\item $\GL_n(\bfZ)$ for $n\ge1$,
\item subgroups of residually finite groups,
\item semidirect products $\Gamma\rtimes \Lambda$ with $\Gamma$ and $\Lambda$ residually finite and $\Gamma$ finitely generated.
\end{enumerate}
\item\label{lem:grp-thy-lem:ii} Fix a group $G$ with finite abelianisation and $\fr(G)=0$, a finitely generated free abelian group $A$, and two central extensions of $G$ of the form,
\[\smash{0\lra A\xlra{\iota_i} E_i\lra G\lra 0\quad\text{for }i=0,1},\]
such that their extension classes in $\oH^2(G;A)$ agree in  $\Hom(\oH_2(G),A)$. Then the finite residuals of $E_0$ and $E_1$ agree in the sense that 
\[\smash{\fr(E_i)\subset \iota_i(A)\quad\text{for }i=0,1\quad\text{and}\quad \iota_0^{-1}(\fr(E_0))=\iota_1^{-1}(\fr(E_1)).}\]
\item\label{lem:grp-thy-lem:iv}  Fix a group $G$ with finitely generated abelianisation and $\fr(G)=0$, a finitely generated abelian group $A$, and a central extension $0\ra A\ra E\ra G\ra 0$ whose extension class in $\oH^2(G;A)$ is trivial in $\Hom(\oH_2(G),A)$. Then the finite residual $\fr(E)$ is trivial.
\end{enumerate}

\end{lem}
\begin{proof}Item \ref{lem:grp-thy-lem:functoriality} follows directly from the definition of the finite residual. To show \ref{lem:grp-thy-lem:iii}, it suffices by \ref{lem:grp-thy-lem:functoriality} to show the inclusion $\varphi^{-1}(\fr(E'))\subset \fr(E)$, i.e.\,that for all  $x\in\varphi^{-1}(\fr(E'))$ and finite index normal subgroups $H\trianglelefteq E$, we have $x\in H$. Since $\varphi(E)\le E'$ and $H\trianglelefteq E$ have finite index, $\varphi(H)\le E'$ has finite index as well, so we can choose a finite index normal subgroup $\overline{H}\trianglelefteq E'$ with $\overline{H}\le \varphi(H)$. Since $\varphi(x)\in\fr(E')\le \overline{H}$, we conclude that $\varphi(x)\in \varphi(H)$, so we indeed have $x\in H$ since $\varphi$ is injective. For \ref{lem:grp-thy-lem:ex}, see e.g.\,\cite[Proposition 3.1.3, 3.1.4, 3.1.6, 3.5 (ii) on p.\,153]{Wilkes}. The claim $\fr(E_i)\subset \iota_i(A)$ in \ref{lem:grp-thy-lem:ii} follows from $\fr(G)=1$ and \ref{lem:grp-thy-lem:functoriality} applied to $E_i \to G$, so we are left to show $\iota_0^{-1}(\fr(E_0))=\iota_1^{-1}(\fr(E_1))$. To do so, we write $k$ for the (finite) order of $\oH_1(G)$ and $\lambda_i\in \oH^2(G;A)$ for the extension class of $E_i$. From the assumption that the $\lambda_i$ agree in $\Hom(\oH_2(G),A)$ together with the naturality of the universal coefficient theorem and the fact that multiplication by $k$ annihilates $\mathrm{Ext}(\oH_1(G),A)$ since it annihilates $\oH_1(G)$, we conclude $k\cdot \lambda_0=k\cdot \lambda_1\in\oH^2(G;A)$. The latter implies that there is a commutative diagram of groups with exact rows of the form
\[\begin{tikzcd}[row sep=0.4cm]
0\rar&A\arrow["k\cdot(-)",d,hook',swap]\rar{\iota_0}&E_0\arrow[d,hook',"\varphi_0"]\rar&G\rar\dar[equal]&0\\
0\rar&A\rar{\bar{\iota}}&\bar{E}\rar&G\rar&0\\
0\rar&A\arrow["k\cdot(-)",u,hook]\rar{\iota_1}&E_1\arrow[u,hook,"\varphi_1",swap]\rar&G\rar\uar[equal]&0
\end{tikzcd}\]
whose middle row is by definition the central extension classified by $k\cdot \lambda_0=k\cdot \lambda_1$. Since $A$ is finitely generated free abelian, the leftmost vertical morphisms are monomorphisms and their images have finite index, so the same holds for the middle vertical morphisms. From \ref{lem:grp-thy-lem:iii} applied to $\varphi_i$ we thus get $\fr(E_i)=\varphi_i^{-1}(\fr(\bar{E}))$ and thus $\iota_i^{-1}(\fr(E_i))=\iota_i^{-1}(\varphi_i^{-1}(\fr(\bar{E})))$ for $i=0,1$. But $\varphi_i\circ\iota_i=\bar{\iota}\circ (k\cdot(-))$ does not depend on $i$, so $\iota_0^{-1}(\fr(E_0))=\iota_1^{-1}(\fr(E_1))$ as claimed.

To show \ref{lem:grp-thy-lem:iv}, note that the assumption that the extension class vanishes in $\Hom(\oH_2(G),A)$ is---by the universal coefficient theorem---equivalent to the class being in the image of the canonical map $\Ext^1(\oH_1(G),A)\ra \oH^2(G;A)$. Under the equivalence between $\oH^2(G;A)$ and isomorphism classes of central extension of $G$ by $A$, this says that the extension $\smash{0\ra A\ra E\ra G\ra 0}$ (classified by the class in $\oH^2(G;A)$) is pulled back from an extension of abelian groups of the form $\smash{0\ra A\ra E'\ra \oH_1(G)\ra 0}$ (classififed by a class in $\Ext^1(\oH_1(G),A)$) along the abelianisation morphism $G\ra \oH_1(G)$. In particular, we have a monomorphism $E\hookrightarrow E'\times G$. Since $A$ and $\oH_1(G)$ are finitely generated abelian by assumption, the same holds for $E'$, so $E'$ is residually finite by the first part of \ref{lem:grp-thy-lem:ex}. Since $G$ is also residually finite, the product $E'\times G$ is too (using the final part of \ref{lem:grp-thy-lem:ex}), so $\fr(E'\times G)=1$. The claim then follows from an application of \ref{lem:grp-thy-lem:functoriality} to the monomorphism $E\hookrightarrow E'\times G$.
\end{proof}

Before turning to the computation of the finite residual $\fr(\Gamma_g)$, we will discuss another preparatory lemma. It involves the \emph{signature class} $\sgn\in\oH^2(\Sp_{2g}(\bfZ);\bfZ)$ (sometimes also denoted $\mathrm{sgn}$) which is a second cohomology class of the symplectic group over the integers, constructed in terms of signatures of certain symmetric bilinear forms that one can associate to a bundle of symplectic forms over surfaces (see e.g.\,\cite[p.\,471]{KRW-arithmetic} for a definition for $g\ge3$; for $g\le 2$ the class is defined via pullback along the inclusions $\Sp_{2g}(\bfZ)\le \Sp_{2g'}(\bfZ) $ for $g\le g'$).

\begin{lem}\label{lem:criterion-z-fr}Fix $g\ge2$, a subgroup $\Lambda\le \Sp_{2g}(\bfZ)$, and a cohomology class $\lambda\in\oH^2(\Lambda;\bfZ)$ such that the following conditions are satisfied:
\begin{enumerate}
\item\label{lem:criterion-z-fr:i} $\Lambda\le \Sp_{2g}(\bfZ)$ has finite index.
\item\label{lem:criterion-z-fr:ii} The abelianisation of $\Lambda$ is finite.
\item\label{lem:criterion-z-fr:iii} The class $8\cdot \lambda\in\oH^2(\Lambda;\bfZ)$ and the pullback of the signature class $\sgn\in\oH^2(\Sp_{2g}(\bfZ);\bfZ)$ have the same image in $\Hom(\oH_2(\Lambda;\bfZ),\bfZ)$. 
\end{enumerate} 
Then the extension $\smash{0\ra \bfZ\xsra{\iota}E(\lambda)\ra \Lambda\ra 0}$ classified by $\lambda$ satisfies $\fr(E(\lambda))=\iota(\bfZ)$. 
\end{lem}
\begin{proof}Consider the maps of extensions
\[\begin{tikzcd}[column sep=0.5cm]
0\rar& \bfZ\rar{\iota_\lambda}\arrow[d,"8\cdot(-)",hook,swap]&E(\lambda)\rar\arrow[d,hook,"\alpha"]& \Lambda\arrow[d,equal]\rar &0\\
0\rar& \bfZ\rar{\iota_{8\cdot\lambda}}&E(8\cdot \lambda)\rar& \Lambda\rar &0
\end{tikzcd}\quad\quad \begin{tikzcd}[column sep=0.5cm]
0\rar& \bfZ\rar{\iota_{\sgn'}}\arrow[d,equal]&E(\sgn')\rar\arrow[d,hook,"
\beta"]& \Lambda\arrow[d,"\subset",hook]\rar &0\\
0\rar& \bfZ\rar{\iota_{\sgn}}&E(\sgn)\rar& \Sp_{2g}(\bfZ)\rar &0
\end{tikzcd}\]
where the left-hand diagram arises from pushing out the top extension along $8\cdot(-)\colon \bfZ\hookrightarrow \bfZ$ and the right-hand diagram arises from pulling back the bottom extension classified by the signature class along the inclusion $\Lambda\le \Sp_{2g}(\bfZ)$. By the discussion in \cite[p.\,472]{KRW-arithmetic}, we have $\fr(E(\sgn))=\iota_{\sgn}(8\cdot \bfZ)$. Applying \cref{lem:grp-thy-lem} \ref{lem:grp-thy-lem:iii} to $\beta$ using assumption \ref{lem:criterion-z-fr:i} yields $\smash{\fr(E(\sgn'))=\beta^{-1}(\iota_{\sgn}(8\cdot \bfZ))=\iota_{\sgn'}(8\cdot \bfZ)}$. Moreover, combining assumptions \ref{lem:criterion-z-fr:ii} and \ref{lem:criterion-z-fr:iii} with the fact that $\fr(\Lambda)=1$ by \cref{lem:grp-thy-lem} \ref{lem:grp-thy-lem:ex}, we may apply \cref{lem:grp-thy-lem} \ref{lem:grp-thy-lem:ii} to $E(8\cdot\lambda)$ and $E(\sgn')$ to conclude $\smash{\iota_{8\cdot\lambda}^{-1}(\fr(E(8\cdot\lambda)))=\iota_{\sgn'}^{-1}(\fr(E(\sgn')))=8\cdot\bfZ}$ and $\smash{\fr(E(8\cdot\lambda))\subset \iota_{8\cdot\lambda}(\bfZ)}$, so we conclude the equality $\smash{\fr(E(8\cdot\lambda))=\iota_{8\cdot\lambda}(8\cdot\bfZ)}$. Finally, from an application of \cref{lem:grp-thy-lem} \ref{lem:grp-thy-lem:iii} to $\alpha$, we get $\smash{\fr(E(\lambda))=\alpha^{-1}(\iota_{8\cdot\lambda}(8\cdot\bfZ))=\iota_\lambda(\bfZ)}$ as claimed. 
\end{proof}

\subsection{Proof of \cref{thm:fr-prelim}}\label{sec:fr-computation}

We fix $g\ge0$ and an odd integer $n\ge3$ throughout and adopt the notation from \cite[Section 1]{Krannich-mcg}. In particular:
\begin{enumerate}[leftmargin=0.6cm,label=(\alph*)]
\item $G_g\le \Sp_{2g}(\bfZ)$ denotes the image of the morphism $p\colon \Gamma_{g}\ra \Sp_{2g}(\bfZ)$ given by the action on $\smash{\oH_{n}(W_{g};\bfZ)\cong\bfZ^{2g}}$. The latter isomorphism involves the choice of a hyperbolic basis with respect to the intersection form on $\oH_{n}(W_{g,1};\bfZ)$, which we fix once and for all. For $n=3,7$, we have $G_g=\Sp_{2g}(\bfZ)$ and for $n\neq3,7$ the subgroup $G_g\le \Sp_{2g}(\bfZ)$ has finite index and agrees with a certain subgroup $\smash{G_g=\Sp_{2g}^q(\bfZ)\le \Sp_{2g}(\bfZ)}$ whose definition involves a quadratic refinement (see e.g.\,Section 1.2 loc.cit.).
\item\label{enum:gamma-ext} Fixing an embedded codimension $0$ disc $D^{2n}\subset W_g$, the groups $\Gamma_{g,1}\coloneq \pi_0(\Diff_\partial(W_{g,1}))$ and $\Gamma_{g,\half}\coloneq\pi_0(\Diff_{\half\partial}(W_{g,1}))$ are the groups of isotopy classes of diffeomorphisms of $W_{g,1}\coloneq (\sharp^g(S^n\times S^n))\backslash \interior(D^{2n})$ that fix pointwise a neighbourhood of the boundary or a fixed embedded codimension $0$ disc in the boundary, respectively. We need a few facts on these groups: Firstly, the morphism $\Gamma_{g,1}\ra \Gamma_g$ induced by extending diffeomorphisms along $W_{g,1}\subset W_g$ by the identity is an isomorphism (see Lemma 1.1 loc.cit.). Secondly, the forgetful morphism $\Gamma_{g,1}\ra \Gamma_{g,\half}$ is surjective and its kernel can be identified with $\Theta_{2n+1}$ via the morphism $\Theta_{2n+1}\cong\pi_0(\Diff_\partial(D^{2n}))\ra \Gamma_{g,1}$ induced by extending diffeomorphisms of a disc $D^{2n}\subset \interior(W_{g,1})$ by the identity (see (1.6) loc.cit.). Thirdly, the morphism $p\colon \Gamma_{g,1}\ra G_g$ factors over $\Gamma_{g,\half}$ (see (1.7) loc.cit.).
\item\label{enum:gamma-semidirect} Fixing a stable framing $F$ of $W_{g,1}$, acting on $F$ yields a morphism $(s_F,p)\colon \Gamma_{g,\half}\rightarrow (\bfZ^{2g}\otimes \pi_n(\SO))\rtimes G_g$ (see p.\,81 loc.cit.). By combining Equation (2.2) loc.cit.\,with Lemma 2.1 loc.cit., this morphism is an isomorphism for $n\neq3,7$, and is for $n= 3,7$ a monomorphism onto a subgroup of finite index. In particular, $\Gamma_{g,\half}$ is isomorphic to a subgroup of $(\bfZ^{2g}\otimes \pi_n(\SO))\rtimes G_g\le (\bfZ^{2g}\otimes \pi_n(\SO))\rtimes \GL_{2g}(\bfZ)$, so $\fr(\Gamma_{g,\half})=1$ as a result of \cref{lem:grp-thy-lem} \ref{lem:grp-thy-lem:ex}.
\end{enumerate}

The main ingredient in the proof of \cref{thm:fr-prelim} is the identification from loc.cit. of the extension of groups resulting from \ref{enum:gamma-ext} of the form
\begin{equation}\label{equ:main-ext}0\lra \Theta_{2n+1}\lra \Gamma_{g,1}\lra \Gamma_{g,\half}\lra 0.\end{equation}
 This identification involves two morphisms
\begin{equation}\label{equ:two-morphisms}\sgn\colon \oH_2(\Sp_{2g}(\bfZ);\bfZ)\lra \bfZ\quad\text{and}\quad\chi^2\colon \oH_2(\bfZ^{2g}\rtimes
\Sp_{2g}(\bfZ);\bfZ)\lra \bfZ,\end{equation}
defined in Sections 3.4-3.5 loc.cit., which enjoy the following properties:
\begin{enumerate}[leftmargin=*]
\item\label{enum:sgn-lift} When pulled back along $\smash{\Sp_{2g}^q(\bfZ)\le  \Sp_{2g}(\bfZ)}$, the first morphism becomes  divisible by $8$ and there is a preferred lift of the resulting morphism $\sgn/8\colon \oH_2(\Sp_{2g}^q(\bfZ);\bfZ)\ra \bfZ$ to a cohomology class (see Definition 3.17 (i) loc.cit.) \begin{equation}\label{equ:sgnlift}\smash{\textstyle{\frac{\sgn}{8}\in \oH^2(\Sp_{2g}^q(\bfZ);\bfZ)}}.\end{equation}
\item When pulled back along the inclusion $\bfZ^{2g}\rtimes\Sp_{2g}^q(\bfZ)\longrightarrow \bfZ^{2g}\rtimes\Sp_{2g}(\bfZ)$, the second morphism becomes divisible by $2$ and there is a preferred lift of the resulting morphism $\chi^2/2\colon \oH_2(\bfZ^{2g}\rtimes\Sp_{2g}^q(\bfZ);\bfZ)\ra \bfZ$ to a class (see Definition 3.20 (i) loc.cit.) \begin{equation}\label{equ:chisquaredlift}\smash{\textstyle{\frac{\chi^2}{2}\in \oH^2(\bfZ^{2g}\rtimes\Sp_{2g}^q(\bfZ);\bfZ)}.}\end{equation}
\item\label{enum:sgnchsquare-lift} When pulled back along the inclusion $(s_F,p)\colon \Gamma_{g,\half}\hookrightarrow \bfZ^{2g}\rtimes\Sp_{2g}(\bfZ)$, the morphism $\chi^2-\sgn\colon \oH_2(\bfZ^{2g}\rtimes
\Sp_{2g}(\bfZ);\bfZ)\ra \bfZ$ becomes for $n=3,7$ divisible by $8$ and there is a preferred lift to the resulting morphism $(\chi^2-\sgn)/8\colon \oH_2(\Gamma_{g,\half};\bfZ)\ra \bfZ$ to a  class \[\smash{\textstyle{\frac{\chi^2-\sgn}{8}\in \oH^2(\Gamma_{g,\half};\bfZ)}}.\]
\item \label{enum:sgn-genusone} The morphism $\sgn$ is induced by the same-named signature class $\sgn\in\oH^2(\Sp_{2g}(\bfZ);\bfZ)$ featuring in \cref{lem:criterion-z-fr} above. For $g=1$, the morphism $\sgn\colon\oH_2(\Sp_{2g}(\bfZ);\bfZ)\ra \bfZ$ is trivial (see Lemma 3.15 loc.cit.) and the lift \eqref{equ:sgnlift} vanishes too (see Lemma 3.21 loc.cit.).

\end{enumerate}

In addition to the properties \ref{enum:sgn-lift}--\ref{enum:sgn-genusone}, we will use the following of the morphism $\chi^2$:
\begin{lem}\label{lem:chi-scaling} Writing $(k,\id)\colon \bfZ^{2g}\rtimes\Sp_{2g}(\bfZ)\rightarrow \bfZ^{2g}\rtimes\Sp_{2g}(\bfZ)$ for the morphism given by multiplication in $\bfZ^{2g}$ with a fixed integer $k\in \bfZ$, the following relations hold:
\begin{enumerate}
\item \label{enum:first-relation-chi}
$\chi^{2}\circ (k,\id)_* =k^2\cdot \chi^2\in\Hom(\oH_2(\bfZ^{2g}\rtimes
\Sp_{2g}(\bfZ);\bfZ),\bfZ)$ and 
\item \label{enum:second-relation-chi}$(k,\id)^*\frac{\chi^2}{2}=k^2\cdot \frac{\chi^2}{2}\in\oH^2(\bfZ^{2g}\rtimes\Sp_{2g}^q(\bfZ);\bfZ).$
\end{enumerate}
In particular, the relation \ref{enum:first-relation-chi} implies for $k=0$ that the precomposition $\chi^2\circ \inc_*\colon \oH_2(\Sp_{2g}(\bfZ);\bfZ)\ra \bfZ$ of $\chi^2$ with the map induced by the inclusion $\smash{\inc\colon \Sp_{2g}(\bfZ)\hookrightarrow \bfZ^{2g}\rtimes\Sp_{2g}(\bfZ)}$ vanishes.
\end{lem}
\begin{proof}We start by recalling the definition of the morphism $\smash{\chi^2\colon \oH_2(\bfZ^{2g}\rtimes
\Sp_{2g}(\bfZ);\bfZ)\ra \bfZ}$ from \cite[Section 3.5]{Krannich-mcg}. Any cohomology class in $\oH_2(\bfZ^{2g}\rtimes\Sp_{2g}(\bfZ);\bfZ)$ can be represented by a continuous map $\smash{f\colon S\ra B(\bfZ^{2g}\rtimes\Sp_{2g}(\bfZ))}$ from an oriented closed connected surface $S$. The composition $\smash{f_2\coloneq (B(\pr_2)\circ f)\colon S\ra \BSp_{2g}(\bfZ)}$ gives rise to a symmetric bilinear form $\smash{\langle-,-\rangle_{f_2}\colon \oH^1(S;f_2^*\bfZ^{2g})^{\otimes 2}\ra \bfZ}$ (see Section 3.4 loc.cit.) and $f_1\coloneq \smash{(\pr_1\circ \pi_1(f))\colon \pi_1(S)\ra \bfZ^{2g}}$ defines a $1$-cocycle, so a class $\smash{[f_1]\in \oH^1(S;f_2^*\bfZ^{2g})}$. In these terms, the morphism $\chi^2\colon \oH_2(\bfZ^{2g}\rtimes
\Sp_{2g}(\bfZ);\bfZ)\ra \bfZ$ is defined to send a class represented by $\smash{f\colon S\ra B(\bfZ^{2g}\rtimes\Sp_{2g}(\bfZ))}$ to $\smash{\langle[f_1],[f_1]\rangle_{f_2}\in\bfZ}$ (see Section 3.5 loc.cit.). Using this, we compute 
\[\chi^2((k,\id)_*([f]))=\chi^2([B(k,\id)\circ f])=\langle [(B(k,\id)\circ f)_1],[(B(k,\id)\circ f)_1]\rangle_{f_2}=\langle k\cdot [f_1],k\cdot [f_1]\rangle_{f_2}\]
for any class $\smash{[f\colon S\ra B(\bfZ^{2g}\rtimes\Sp_{2g}(\bfZ))]\in \oH^2(\bfZ^{2g}\rtimes\Sp_{2g}(\bfZ);\bfZ)}$. Here the first equality holds by the definition of $\chi^2$ as just recalled, the second equality holds by observing that $(B(k,\id)\circ f)_2=f_2$ since $\pr_2\circ (k,\id)=\pr_2$, and the third equality holds since $(B(k,\id)\circ f)_1=\pr_1\circ \pi_1(B(k,\id)\circ f)=\pr_1\circ (k,\id)\circ \pi_1(f)=(k\cdot-)\circ \pr_1\circ \pi_1(f)=k\cdot f_1$. Since $\langle -,-\rangle_{f_2}$ is bilinear, we have $\chi^2((k,\id)_*([f]))=\langle k\cdot [f_1],k\cdot [f_1]\rangle_{f_2}=k^2\cdot \langle [f_1],[f_1]\rangle_{f_2}=k^2\cdot \chi^2([f])$, which proves the relation in \ref{enum:first-relation-chi}. To establish the relation \ref{enum:second-relation-chi}, we first recall from Definition 3.20 (i) loc.cit. the lift of the morphism $\smash{(\chi^2\circ\inc_*)/2\in\Hom(\oH_2(\bfZ^{2g}\rtimes\Sp_{2g}^q(\bfZ);\bfZ),\bfZ)}$ to a cohomology class $\smash{\frac{\chi^2}{2}\in\oH^2(\bfZ^{2g}\rtimes\Sp_{2g}^q(\bfZ);\bfZ)}$. For $g\ge3$, there is a certain morphism $\smash{a\colon \bfZ/4\bfZ\ra \Sp_{2g}^q(\bfZ)\le \bfZ^{2g}\rtimes\Sp_{2g}^q(\bfZ)}$ defined in Section 3.4 loc.cit.\,which features in a splitting of the form (see p.\,81 loc.cit.)
\begin{equation}\label{equ:chi2-splitting}\smash{a^*\oplus h\colon \oH^2(\bfZ^{2g}\rtimes\Sp_{2g}^q(\bfZ);\bfZ)\xlra{\cong} \oH^2(\bfZ/4\bfZ)\oplus\Hom(\oH_2(\bfZ^{2g}\rtimes\Sp_{2g}^q(\bfZ)),\bfZ)}\end{equation}
where $h$ is the map from the universal coefficient theorem. For $g\ge 3$, the lift $\smash{\frac{\chi^2}{2}}$ is defined as the inverse of $(0,\chi^2/2)$ under this isomorphism, and for $g\le 2$ it is defined by pulling back $\smash{\frac{\chi^2}{2}}$ along the inclusion $\bfZ^{2g}\rtimes\Sp_{2g}^q(\bfZ)\le \bfZ^{2g'}\rtimes\Sp_{2g'}^q(\bfZ)$ for $g'\gg0$. Since $\smash{a\colon \bfZ/4\ra \bfZ^{2g}\rtimes\Sp_{2g}^q(\bfZ)}$ factors through the inclusion $\smash{\Sp_{2g}(\bfZ)\le \bfZ^{2g}\rtimes\Sp_{2g}(\bfZ)}$, it commutes with the morphism $\smash{(k,\id)\colon \bfZ^{2g}\rtimes\Sp_{2g}(\bfZ)\rightarrow \bfZ^{2g}\rtimes\Sp_{2g}(\bfZ)}$, so the induced map is with respect to the splitting \eqref{equ:chi2-splitting} given by the identity on the first summand and by precomposition with $(k,id)_*$ on the second. Hence, since $\chi^2\circ(k,\id)_*=k^2$ by the first relation \ref{enum:first-relation-chi}, we can deduce the claimed relation $\smash{(k,\id)^*\frac{\chi^2}{2}=k^2\cdot \frac{\chi^2}{2}}$.
\end{proof}

The identification of the extension \eqref{equ:main-ext} in \cite{Krannich-mcg} that we will rely on also involves two exotic spheres $\Sigma_P,\Sigma_Q\in\Theta_{2n+1}$, which are defined in Section 3.2.1 loc.cit.. All we need to know about them is that $\Sigma_P$ generates the subgroup $\bP_{2n+2}\le \Theta_{2n+1}$. In terms of the pullbacks of the cohomology classes in \ref{enum:sgn-lift}-\ref{enum:sgnchsquare-lift} above to $\Gamma_{g,\half}$ and the two exotic spheres, the extension \eqref{equ:main-ext} is classified by the class (see Theorem 3.22 and Lemma 3.4 loc.cit.)
\begin{equation}\label{equ:extension-class}
\begin{cases}
\frac{\sgn}{8}\cdot \Sigma_P&n\equiv 1(4)\\
\frac{\sgn}{8}\cdot \Sigma_P+\frac{\chi^2}{2}\cdot \Sigma_Q&n\equiv 3(4)\text{ and }n\neq 3,7\\
-\frac{\chi^2-\sgn}{8}\cdot\Sigma_P&n=3,7
\end{cases}\quad \in\oH^2(\Gamma_{g,\half};\Theta_{2n+1}).\end{equation}
We now turn to the proof of \cref{thm:fr-prelim}, which is divided into four sub-claims. We write \[\mu\coloneq\begin{cases}
\frac{\sgn}{8}&n\neq3,7\\
-\frac{\chi^2-\sgn}{8}&n=3,7
\end{cases}\quad \in\oH^2(\Gamma_{g,\half};\bfZ)\]
and write $0\ra \Theta_{2n+1}\ra E(\mu\cdot\Sigma_P)\ra \Gamma_{g,\half}\ra 0$ for the extension classified by $\mu\cdot\Sigma_P$.

\medskip

\noindent \textbf{Claim \circled{1}}: We have $\fr(\Gamma_{g,1})\cap\Theta_{2n+1}=\fr(E(\mu\cdot \Sigma_P))\cap\Theta_{2n+1}$.
\begin{proof}For $n\equiv 1(4)$ or $n=3,7$, we have $\Gamma_{g,1}=E(\mu\cdot \Sigma_P)$ by \eqref{equ:extension-class}, so there is nothing to show. We thus assume $n\equiv 3(4)$ and $n\neq 3,7$. Fixing an isomorphism $\pi_n(\oO)\cong\bfZ$, we define a self-monomorphism $\kappa\colon \Gamma_{g,\half}\hookrightarrow \Gamma_{g,\half}$ by the following commutative square whose horizontal arrows are isomorphisms by \ref{enum:gamma-semidirect} and whose left vertical map is induced by multiplication by $k$ in $\bfZ^{2g}$ for some fixed number $k\neq0$
\begin{equation}\label{equ:dfn-of-kappa}
\begin{tikzcd}[column sep=0.6cm, row sep=0.5cm]
\Gamma_{g,\half}\arrow["{(s_F,p)}","\cong"',r]\arrow["\kappa",d,hook,swap]&\bfZ^{2g}\rtimes \Sp_{2g}^q(\bfZ)\dar[d,"{(k,\id)}",hook]\\
\Gamma_{g,\half}\arrow["{(s_F,p)}","\cong"',r]&\bfZ^{2g}\rtimes \Sp_{2g}^q(\bfZ).
\end{tikzcd}
\end{equation}
 We now consider the map of extensions
\begin{equation}\label{equ:ext-map-case-i}
\begin{tikzcd}[row sep=0.4cm]
0\rar&\Theta_{2n+1}\rar\dar[equal]&\kappa^*\Gamma_{g,1}\rar\arrow[d,hook]&\Gamma_{g,\half}\rar\arrow[d,hook,"{\kappa}"]&0\\
0\rar&\Theta_{2n+1}\rar&\Gamma_{g,1}\rar&\Gamma_{g,\half}\rar&0
\end{tikzcd}
\end{equation}
obtained by pulling back the extension \eqref{equ:main-ext} along $\kappa$. From \cref{lem:grp-thy-lem} \ref{lem:grp-thy-lem:iii}, we obtain $\fr(\Gamma_{g,1})\cap\Theta_{2n+1}=\fr(\kappa^*\Gamma_{g,1})\cap\Theta_{2n+1}$, so it suffices to show that there exists $k\neq0$ such that the top extension in \eqref{equ:ext-map-case-i} is classified by $\mu=\frac{\sgn}{8}\cdot \Sigma_P$. Since the bottom extension is classified by $\smash{\frac{\sgn}{8}+\frac{\chi^2}{2}\in \oH^2(\Gamma_{g,\half};\Theta_{2n+1})}$ by \eqref{equ:extension-class}, this is equivalent to showing that
\begin{equation}\label{equ:cohomology-id-case-i}\smash{\textstyle{\kappa^*(\frac{\sgn}{8})\cdot \Sigma_P+\kappa^*(\frac{\chi^2}{2})\cdot \Sigma_Q=\frac{\sgn}{8} \in\oH^2(\Gamma_{g,\half};\Theta_{2n+1})}}.\end{equation}
Since the endomorphism $(k,\id)$ in \eqref{equ:dfn-of-kappa} commutes with the projection to $\smash{\Sp_{2g}^q(\bfZ)}$, the same applies to $\kappa$ in \eqref{equ:dfn-of-kappa}, so as $\smash{\frac{\sgn}{8}}\in\oH^2(\Gamma_{g,\half};\Theta_{2n+1})$ is pulled back along $p\colon\Gamma_{g,\half} \ra\Sp_{2g}^q(\bfZ)$, this implies that $\smash{\kappa^*(\frac{\sgn}{8})=\frac{\sgn}{8}}$, so it remains to show that $\smash{\kappa^*(\frac{\chi^2}{2})\cdot \Sigma_Q=0}$ in $\oH^2(\Gamma_{g,\half};\Theta_{2n+1})$ for some $k\neq0$. Since the class $\smash{\frac{\chi^2}{2}}$ is pulled back along the isomorphism $(s_F,p)$ in \eqref{equ:dfn-of-kappa}, it suffices to show that $\smash{(k,\id)^*(\frac{\chi^2}{2})\cdot \Sigma_Q=0}$ in $\oH^2(\bfZ^{2g}\rtimes\Sp^q_{2g}(\bfZ);\Theta_{2n+1})$. But in view of \cref{lem:chi-scaling}, we have $\smash{(k,\id)^*(\frac{\chi^2}{2})=k^2\cdot \frac{\chi^2}{2}}$, so if we choose $k$ to be any multiple of the order of $\Sigma_Q$ in $\Theta_{2n+1}$, then we can conclude that $\smash{(k,\id)^*(\frac{\chi^2}{2})\cdot \Sigma_Q= \frac{\chi^2}{2}\cdot k^2\cdot \Sigma_Q=0}$ holds as claimed. 
\end{proof}

\noindent \textbf{Claim \circled{2}}: We have $\fr(\Gamma_{g,1})\subset\bP_{2n+2}$. 
\begin{proof}
Since $\fr(\Gamma_{g,\half})=1$ by \ref{enum:gamma-semidirect}, we have
 $\fr(\Gamma_{g,1})\subset\Theta_{2n+1}$ by applying \cref{lem:grp-thy-lem} \ref{lem:grp-thy-lem:functoriality} to the morphism $\Gamma_{g,1}\ra \Gamma_{g,\half}$ in \eqref{equ:main-ext} with kernel $\Theta_{2n+1}$. In particular, we have $\fr(\Gamma_{g,1})=\fr(\Gamma_{g,1})\cap \Theta_{2n+1}$, so together with Claim \circled{1}, we get $\fr(\Gamma_{g,1})=\fr(E(\mu\cdot\Sigma_P))\cap\Theta_{2n+1}$. Now consider the map of central extensions 
\[
\begin{tikzcd}[row sep=0.4cm]
0\rar&\Theta_{2n+1}\rar\arrow[d,two heads]&E(\mu\cdot\Sigma_P)\rar\arrow[d,two heads]&\Gamma_{g,\half}\rar\arrow[d,equal]&0\\
0\rar&\Theta_{2n+1}/\bP_{2n+2}\rar&E(\mu\cdot\Sigma_P)/\bP_{2n+2}\rar&\Gamma_{g,\half}\rar&0
\end{tikzcd}
\]
 where the bottom extension is obtained from the top extension by taking quotients by the subgroup $\bP_{2n+2}\le \Theta_{2n+1}$. This extension is classified by the image of $\mu\cdot\Sigma_P$ under the morphism $\oH^2(\Gamma_{g,\half};\Theta_{2n+1})\ra \oH^2(\Gamma_{g,\half};\Theta_{2n+1}/\bP_{2n+2})$ induced by the quotient map $\Theta_{2n+1}\ra \Theta_{2n+1}/\bP_{2n+2}$. Since $\Sigma_P$ is contained in $\bP_{2n+2}$, this image is trivial, so it follows that $E(\mu\cdot\Sigma_P)/\bP_{2n+2}\cong \Theta_{2n+1}\times \Gamma_{g,\half}$. Thus, since $\fr(\Gamma_{g,1})=1$, we deduce from \cref{lem:grp-thy-lem} \ref{lem:grp-thy-lem:ex} that $\fr(E(\mu\cdot\Sigma_P)/\bP_{2n+2})=1$, so applying \cref{lem:grp-thy-lem} \ref{lem:grp-thy-lem:functoriality} to to $E(\mu\cdot\Sigma_P)\ra E(\mu\cdot\Sigma_P)/\bP_{2n+2}$ we conclude that $\fr(\mu\cdot\Sigma_P)\subset \bP_{2n+2}$. Combining this with the first part of the argument, we arrive at the claim $\fr(\Gamma_{g,1})=\fr(E(\mu\cdot\Sigma_P))\cap\Theta_{2n+1}\subset \bP_{2n+2}$.
\end{proof}

\noindent \textbf{Claim \circled{3}}: For $g\le 1$, we have $\fr(\Gamma_{g,1})=1$. 
\begin{proof}For $g=0$ we have $\Gamma_{g,1}=\pi_0\Diff_\partial(D^{2n})\cong\Theta_{2n+1}$, so since $\Theta_{2n+1}$ is a finite abelian group, we have $\fr(\Gamma_{g,1})=1$ by \cref{lem:grp-thy-lem} \ref{lem:grp-thy-lem:ex}, as claimed. In the case $g=1$ and $n\neq3,7$, we use that $\smash{\frac{\sgn}{8}}$ is trivial for $g=1$ (see \ref{enum:sgn-genusone}), which implies $\smash{E(\frac{\sgn}{8}\cdot\Sigma_P)\cong \Gamma_{g,\half}\times \Theta_{2n+1}}$ and thus $\fr(E(\frac{\sgn}{8}\cdot\Sigma_P))=1$ by \cref{lem:grp-thy-lem} \ref{lem:grp-thy-lem:ex} and the fact from \ref{enum:gamma-semidirect} that $\fr(\Gamma_{g,\half})=1$. Combining this with Claim \circled{1}, we obtain  $\fr(\Gamma_{g,1})\cap\Theta_{2n+1}=1$. But since $\fr(\Gamma_{g,1})\subset\Theta_{2n+1}$ by Claim $\circled{2}$, we arrive at the claim $\fr(\Gamma_{g})=1$. In the remaining case $n=3,7$ and $g=1$, we first consider the map of extensions from \cite[Equation (2.2)]{Krannich-mcg} which involves the monomorphism $(s_F,p)$ from \ref{enum:gamma-semidirect}
\[
\begin{tikzcd}[row sep=0.4cm, column sep=0.2cm]
0\rar&\bfZ^{2g}\otimes \mathrm{im}(\pi_n(\SO(n)\ra\pi_n(\SO(n+1)))\rar\arrow[d,"\id\otimes\inc"hook]&\Gamma_g\rar{p}\arrow[d,"{(s_F,p)}", hook]&\Sp_{2g}(\bfZ)\arrow[d,equal]\rar&0\\
0\rar&\bfZ^{2g}\otimes \pi_n(\SO)\rar&(\bfZ^{2g}\otimes \pi_n(\SO))\rtimes\Sp_{2g}(\bfZ) \rar&\Sp_{2g}(\bfZ)\rar&0
\end{tikzcd}
\]
By Lemma 2.1, the image of the monomorphism $\mathrm{im}(\pi_n(\SO(n)\ra\pi_n(\SO(n+1)))\hookrightarrow \pi_n(\SO)\cong\bfZ$ has index $2$. By Theorem 2.2, the upper extension is trivial, so we can conclude that there is an isomorphism $\smash{\Gamma_{g,\half}\cong \bfZ^{2g}\rtimes \Sp_{2g}(\bfZ)}$ with respect to which the monomorphism $\smash{(s_F,p)\colon \Gamma_{g,\half}\hookrightarrow \bfZ^{2g}\rtimes \Sp_{2g}(\bfZ)}$ from \ref{enum:gamma-semidirect} is given by $\smash{(2,\id)\colon \bfZ^{2g}\rtimes \Sp_{2g}(\bfZ)\hookrightarrow \bfZ^{2g}\rtimes \Sp_{2g}(\bfZ)}$. We can now define a self-monomorphism $\kappa\colon \Gamma_{g,\half}\hookrightarrow\Gamma_{g,\half}$ for any fixed number $k\neq 0$ similar to \eqref{equ:dfn-of-kappa}, but using the above isomorphism $\smash{\Gamma_{g,\half}\cong \bfZ^{2g}\rtimes \Sp_{2g}(\bfZ)}$ instead of $(s_F,p)$. Arguing as in the proof of Claim $\circled{1}$ we have $\smash{\fr(\Gamma_{g,1})\cap\Theta_{2n+1}=\fr(\kappa^*\Gamma_{g,1})\cap\Theta_{2n+1}}$ which together with Claim $\circled{2}$ yields $\fr(\Gamma_{g,1})=\fr(\kappa^*\Gamma_{g,1})\cap\Theta_{2n+1}$, so it would suffice to show $\fr(\kappa^*\Gamma_{g,1})=1$. This follows from an application of \cref{lem:grp-thy-lem} \ref{lem:grp-thy-lem:iv} to the extension central extension $\smash{0\ra \Theta_{2n+1}\ra \kappa^*\Gamma_{g,1}\ra \Gamma_{g,1\half}\ra0}$ once we verify the three assumptions that (I) $\Gamma_{g,1\half}$ has finitely generated abelianisation, (II) $\fr(\Gamma_{g,1\half})=1$, and (III) the cohomology class $\smash{\kappa^*(-\frac{\chi^2-\sgn}{8})\cdot\Sigma_P\in\oH^2(\Gamma_{g,1\half};\Theta_{2n+1})}$ which classifies this extension by \eqref{equ:extension-class} has trivial image in $\Hom(\oH_2(\Gamma_{g,1\half}),\Theta_{2n+1})$ for some choice of $k\neq0$. Condition (I) holds since $\Gamma_{g,1\half}$ is itself finitely generated, because we saw above that it is is isomorphic to $\bfZ^{2g}\rtimes \Sp_{2g}(\bfZ)$. Condition (II) holds by \ref{enum:gamma-semidirect}. This leaves us with checking condition (III) for which it suffices to show that  $ \smash{\kappa^*(-\frac{\chi^2-\sgn}{8})}$ maps under the canonical morphism $h\colon \oH^2(\Gamma_{g,1\half});\bfZ)\ra \Hom\oH_2(\Gamma_{g,1\half}),\bfZ)$ to $\smash{-k^2\cdot (\chi^2-\sgn)/8\circ (s_F,p)_*}$, since then $\smash{\kappa^*(-\frac{\chi^2-\sgn}{8})\cdot\Sigma_P}$ maps to $-k^2\cdot \Sigma_P\cdot ((\chi^2-\sgn)/8\circ (s_F,p)_*)$ in $\Hom(\oH_2(\Gamma_{g,1\half}),\Theta_{2n+1})$ which is trivial if we choose $k\neq0$ to divide the order of $\Sigma_P$. To show this remaining identity $h( \smash{\kappa^*(-\frac{\chi^2-\sgn}{8})})=\smash{-k^2\cdot (\chi^2-\sgn)/8\circ (s_F,p)_*}$, we compute 
\[\textstyle{8\cdot h( \smash{\kappa^*(-\frac{\chi^2-\sgn}{8})})=-(\chi^2-\sgn)\circ (s_F,p)_*\circ \kappa_*=-(\chi^2-\sgn)\circ (k,\id)_*\circ (s_F,p)_*}\]
where the first equality follows from \ref{enum:sgnchsquare-lift} and the second equality holds since $(s_F,p)\circ \kappa=(k,\id)\circ (s_F,p)$ by definition of $\kappa$. We then compute further
\[-(\chi^2-\sgn)\circ (k,\id)_*\circ (s_F,p)_*=-\chi^2\circ (k,\id)_*\circ (s_F,p)_*=-k^2\cdot \chi^2\circ (s_F,p)
\]
using the final part of \ref{enum:sgn-genusone} for the first equality and \cref{lem:chi-scaling} for the second. In total we obtain $8\cdot h( \smash{\kappa^*(-\frac{\chi^2-\sgn}{8})})=-k^2\cdot \chi^2\circ (s_F,p)$ in $\Hom(\oH_2(\Gamma_{g,1\half}),\bfZ)$, so the claimed identity follows by dividing this equality by $8$. 
\end{proof}

\noindent \textbf{Claim \circled{4}}: For $g\ge 2$, we have $\bP_{2n+2}\subset\fr(\Gamma_{g,1})$, and thus $\bP_{2n+2}=\fr(\Gamma_{g,1})$ by Claim $\circled{2}$.

\begin{proof} We first form the pullback involving the monomorphism $(s_F,p)$ from \ref{enum:gamma-semidirect}
\begin{equation}\label{equ:dfn-gammaF}
\begin{tikzcd}
\Gamma^F_{g,\half}\arrow[r,hook,"p_F",swap]\arrow[dr, phantom, "\scalebox{1}{$\lrcorner$}" , pos=0, color=black]\arrow[d,hook,"{\iota_F}"]&\Sp_{2g}(\bfZ)\arrow[d,hook,"\inc"]\\
\Gamma_{g,\half}\arrow[r,hook,"{(s_F,p)}",swap]& (\bfZ^{2g}\otimes \pi_n(\oO))\rtimes \Sp_{2g}(\bfZ)
\end{tikzcd}\end{equation}
and consider the diagram of horizontal of extensions
\begin{equation}\label{equ:iota-f-mu-ext}
\begin{tikzcd}[row sep=0.1cm,column sep=0.1cm]
&0\arrow[rr]&&[20pt]\Theta_{2n+1}\arrow[dd,equal]\arrow[rr]&&E(\iota_F^*\mu\cdot\Sigma_P)\arrow[dd,hook]\arrow[rr] &&\Gamma^F_{g,\half}\arrow[rr]\arrow[dd,hook,"\iota_F",near start]&&0\\
0\arrow[rr]&&\bfZ\arrow[ur,"\Sigma_P",swap]\arrow[rr,crossing over]\arrow[dd,equal,crossing over]&&E(\iota_F^*\mu)\arrow[ur]\arrow[rr,crossing over]\arrow[dd,hook,crossing over] &&\Gamma^F_{g,\half}\arrow[rr,crossing over]\arrow[dd,hook,"\iota_F",near start,crossing over]\arrow[ur,equal]&&0\\
&0\arrow[rr]&&\Theta_{2n+1}\arrow[rr]&&E(\mu\cdot\Sigma_P)\arrow[rr] &&\Gamma_{g,\half}\arrow[rr]&&0\\
0\arrow[rr]&&\bfZ\arrow[ur,"\Sigma_P",swap]\arrow[rr]&&E(\mu)\arrow[rr]\arrow[ur] &&\Gamma_{g,\half}\arrow[rr]\arrow[ur,equal]&&0
\end{tikzcd}
\end{equation}
obtained from the front bottom row classified by $\mu\in \oH^{2}(\Gamma_{g,\half};\bfZ)$ by pulling it back along the inclusion $\iota_F$ to obtain the front part of the diagram, and then pushing out along $\Sigma_P\colon \bfZ\ra \Theta_{2n+1}$ to obtain the back part. In a moment, we will show that $\smash{\fr(E(\iota_F^*\mu))=\bfZ}$. Assuming this for now, the proof concludes as follows: Applying \cref{lem:grp-thy-lem} \ref{lem:grp-thy-lem:functoriality} to the map $E(\iota_F^*\mu)\ra E(\mu\cdot \Sigma_P)$, we deduce $\langle \Sigma_P\rangle \subset \fr(E(\mu\cdot\Sigma_P))$, so since $\bP_{2n+2}=\langle \Sigma_P\rangle $, we obtain $\bP_{2n+2}\subset \fr(E(\mu\cdot\Sigma_P))$ which combined with Claim $\circled{1}$ implies $\bP_{2n+2}\subset \fr(\Gamma_{g,1})$ as claimed. We are thus left to show that $\smash{\fr(E(\iota_F^*\mu))=\bfZ}$ which we do by verifying the assumptions of \cref{lem:criterion-z-fr} for the class $\smash{\iota_F^*\mu\in \oH^{2}(\Gamma_{g,\half}^F;\bfZ)}$, viewing $\smash{\Gamma_{g,\half}^F}$ as a subgroup of $\smash{\Sp_{2g}(\bfZ)}$ via the inclusion $p_F$ in \eqref{equ:dfn-gammaF}. Since the index of the image of $p_F$ is the index of the image of $(s_F,p)$ in \eqref{equ:dfn-gammaF} which is finite by \cref{enum:gamma-semidirect}, so the first condition in \cref{lem:criterion-z-fr} holds for $\iota_F^*\mu$. To show the second condition, i.e.\,that the abelianisation $\smash{\oH_1(\Gamma_{g,\half}^F)}$ is finite, note that since the right vertical inclusion in \eqref{equ:dfn-gammaF} is a split injection, the same applies for the left vertical inclusion, so $\smash{\oH_1(\Gamma_{g,\half}^F)}$ is a summand of $\smash{\oH_1(\Gamma_{g,\half})}$ and thus it suffices to prove that $\smash{\oH_1(\Gamma_{g,\half})}$ is finite. This holds by a combination of Corollary 2.4 and Table 2 in loc.cit., using the assumption that $g\ge 2$. This leaves us with establishing the third condition in \cref{lem:criterion-z-fr} for the class $\iota_F^*\mu$, i.e.\,that we have $8\cdot \mu \circ (\iota_F)_*=\sgn\circ (p_F)_*$ in $\Hom(\oH_2(\Gamma_{g,\half}^F),\bfZ)$. From the definition of $\mu$ and commutativity of \eqref{equ:dfn-gammaF}, we see that the left hand side of the claimed equation is given by the composition
\[\smash{\oH_2(\Gamma_{g,\half}^F)\xra{(p_F)_*}\oH_2(\Sp_{2g}(\bfZ))\xra{\inc_*}\oH_2((\bfZ^{2g}\otimes\pi_n(\oO))\rtimes \Sp_{2g}(\bfZ))\lra
\bfZ} \]
where the final arrow is $\smash{
\sgn\circ (\pr_2)_*}$ if $n\neq3,7$ and $\smash{-\chi^2+(\sgn\circ (\pr_2)_*)}$ if $n=3,7$. As $\smash{\chi^2\circ \inc_*}$ vanishes by \cref{lem:chi-scaling}, we have $8\cdot \mu \circ (\iota_F)_*=\sgn\circ (\pr_2)_*\circ \inc_*\circ (p_F)_*$ in both cases. But since $\smash{\pr_2\circ \inc=\id_{\Sp_{2g}(\bfZ)}}$, this implies $8\cdot \mu \circ (\iota_F)_*=\sigma\circ (p_F)_*$ as claimed.\end{proof}

Combining Claims \circled{3} and \circled{4}, we conclude \cref{thm:fr-prelim}.

\bibliographystyle{amsalpha}
\bibliography{literature}

\medskip

\end{document}